\documentclass{amsart}

\usepackage{setspace}
\usepackage{amsmath,amsfonts,amscd,amssymb}
\usepackage{graphicx}
\usepackage[all]{xy}

\newcommand{\mM}{\mathcal{M}}
\newcommand{\bM}{\overline{\mathcal{M}}}

\newcommand{\mJ}{\mathcal{J}}
\newcommand{\mC}{\mathcal{C}}

\newcommand{\mE}{\mathcal{E}}
\newcommand{\mQ}{\mathcal{Q}}
\newcommand{\mB}{\mathcal{B}}
\newcommand{\aA}{\mathbb{A}}
\newcommand{\DD}{\mathbb{D}}
\newcommand{\ZZ}{\mathbb{Z}}

\newcommand{\CC}{\mathbb{C}}
\newcommand{\RR}{\mathbb{R}}
\newcommand{\PP}[1]{\mathbb{CP}^{#1}}
\newcommand{\id}{\mbox{id}}

\newcommand{\Sympl}{\mathbb{S}}

\newcommand{\vol}{\mbox{vol}}
\newcommand{\End}{\mbox{End}}
\newcommand{\dbar}[1]{\overline{\partial}_{#1}}
\newcommand{\cov}{\mbox{cov}}
\newcommand{\ev}{\mbox{ev}}

\newcommand{\dom}{\mbox{dom}}
\newcommand{\pr}{\mbox{pr}}

\newtheorem{thm}{Theorem}[section]
\newtheorem{dfn}[thm]{Definition}
\newtheorem{cor}[thm]{Corollary}
\newtheorem{lma}[thm]{Lemma}
\newtheorem{prp}[thm]{Proposition}
\newtheorem{rmk}[thm]{Remark}

\title[Lagrangian spheres in Del Pezzo surfaces]%
{Lagrangian spheres in Del Pezzo surfaces}
\author{J. D. Evans}
\address{DPMMS, Centre for Mathematical Sciences, Wilberforce Road, Cambridge, CB3 0WA}
\email{j.evans@dpmms.cam.ac.uk}

\begin{document}
\begin{abstract}
Lagrangian spheres in the symplectic Del Pezzo surfaces arising as blow-ups of $\PP{2}$ in 4 or fewer points are classified up to Lagrangian isotopy. Unlike the case of the 5-point blow-up, there is no Lagrangian knotting.
\end{abstract}
\maketitle
\tableofcontents

\section{Introduction}

In a symplectic manifold there is a distinguished class of half-dimensional submanifolds on which the symplectic form vanishes. These are the Lagrangian submanifolds. A pair of Lagrangian submanifolds $L$ and $L'$ which are smoothly isotopic but not isotopic through Lagrangian embeddings is said to be \emph{knotted}. A long-standing problem in symplectic topology is to understand Lagrangian knotting. Progress has been made in the case of Lagrangian 2-spheres in symplectic 4-manifolds, as the following examples illustrate:

\begin{thm}[Hind, \cite{Hi03}]\label{Hindthm}
Let $U$ be a neighbourhood of the zero-section in the cotangent bundle of $S^2$, equipped with its canonical symplectic form. Any Lagrangian 2-sphere in $U$ can be isotoped through embedded Lagrangian 2-spheres until it coincides with the zero-section.
\end{thm}

\begin{thm}[Seidel, \cite{Sei00}]
Let $A_{k+1}$ be the plumbing of $k>1$ copies of the cotangent bundle of $S^2$. In each homology class containing a Lagrangian 2-sphere there are infinitely many smoothly isotopic Lagrangian 2-spheres which are pairwise knotted.
\end{thm}

Seidel's examples are constructed from the zero-sections by a process of iterated Dehn twisting. A Dehn twist is a symplectomorphism supported in a neighbourhood of a Lagrangian sphere (see \cite{Sei08}) and Seidel's construction relies on the fact that the square of a Dehn twist is isotopic to the identity smoothly but not through symplectomorphisms. Actually showing that the images of the zero-sections under iterated Dehn twists are non-isotopic involves calculating Floer homology groups.

Hind's result is a beautiful application of the theory of punctured holomorphic curves. In the preprint \cite{Hi03}, he also claims that any Lagrangian sphere in Seidel's $A_3$ space is isotopic to one obtained by iterated Dehn twisting of a zero section.

The present paper studies the isotopy question for Lagrangian 2-spheres in \emph{symplectic Del Pezzo surfaces}.

\begin{dfn}
A symplectic Del Pezzo surface is the symplectic manifold underlying a smooth complex projective surface with ample anticanonical line bundle. Equivalently it is one of the following symplectic 4-manifolds:

\begin{itemize}
\item $Q=S^2\times S^2$ equipped with the product of the area-1 Fubini-Study forms on each factor (this also arises as the K\"{a}hler form on a smooth quadric hypersurface in $\PP{3}$),
\item A symplectic blow-up $\DD_n$ of $\PP{2}$ (with its anticanonical symplectic form $3\omega_{FS}$) in $n<9$ symplectic balls of equal volume such that
\[\omega_{\DD_n}(E_k)=\omega_{FS}(H)\]
\noindent where $E_k$ is an exceptional sphere and $H$ is a line in $\PP{2}$.
\end{itemize}

Note that all these manifolds are monotone and support a complex structure for which the monotone symplectic form is K\"{a}hler.
\end{dfn}

Let us review what is known about Lagrangian 2-spheres in Del Pezzo surfaces:

\begin{itemize}
\item $\PP{2}$, $\DD_1$: Neither of these spaces contains an embedded Lagrangian 2-sphere. This is clear for homology reasons: by Weinstein's neighbourhood theorem any Lagrangian 2-sphere has self-intersection $-2$, but there is no such homology class in either of these spaces.
\item $Q$: This space contains a unique Lagrangian 2-sphere up to isotopy. This theorem was also proved by Hind \cite{Hi04}.
\item $\DD_2$, $\DD_3$, $\DD_4$: It is easy to find Lagrangian spheres in these spaces, but Seidel has shown that the corresponding squared Dehn twists are Hamiltonian-isotopic to the identity, so there is no obvious way to produce knotting.
\item $\DD_5$, $\DD_6$, $\DD_7$, $\DD_8$: Again, it is easy to find Lagrangian 2-spheres in these spaces. It follows from some results of Seidel on the symplectic mapping class groups of these spaces \cite{Sei08} that knotting occurs.
\end{itemize}

The main theorem of the present paper is the following:

\begin{thm}\label{bigthm}
There is no Lagrangian knotting in $\DD_2$, $\DD_3$ or $\DD_4$.
\end{thm}

\begin{rmk}
It is possible to find $A_3$ and $A_4$ configurations of Lagrangian spheres inside $\DD_3$ and $\DD_4$ respectively and a corollary of the theorem is that although two homologous Lagrangian spheres in these configurations are isotopic, the isotopy must pass through Lagrangian spheres which leave the $A_n$-neighbourhoods.
\end{rmk}

The theorem is proved by first showing that one can find a Lagrangian isotopy taking any 2-sphere $L$ to another sphere $L'$ which is disjoint from some set of divisors. The divisors are chosen so that their complement is symplectomorphic to (a compact subset of) $T^*S^2$, whereupon Hind's theorem guarantees that any two such spheres are Lagrangian isotopic. Achieving disjointness from divisors requires the technology of symplectic field theory and takes up most of the paper.

The idea behind proving disjointness is to find a family of almost complex structures $\{J_t\}_{t=0}^T$ and a family $C_t$ of $J_t$-holomorphic curves representing the relevant configuration of divisors such that:
\begin{itemize}
\item $J_0$ is the standard complex structure and $C_0$ is the standard configuration of divisors,
\item $C_T$ is disjoint from $L$.
\end{itemize}
Then it is not hard (see section \ref{isoprf}) to construct a disjoining isotopy of $L$ from $C_0$. The almost complex structures $J_t$ are obtained by a process called ``stretching the neck'' around $L$ (see diagram \ref{nstretch} below). The $J_t$-holomorphic curves are proven to exist in section \ref{singularpseudo}. The behaviour of $C_t$ as $t\rightarrow\infty$ is analysed using symplectic field theory and it is shown in sections \ref{prf} and \ref{prf2} that for large $t$, $C_t$ must be disjoint from $L$.
\begin{equation}\label{nstretch}
\xy
(0,10)*{};(0,-10)*{} **\crv{(-20,0)};
(0,10)*{};(0,-10)*{} **\crv{(20,20) & (40,0) & (20,-20)};
(0,7)*{};(0,-7)*{} **\crv{(-5,0)};
(20,0)*{};(5,0)*{} **\crv{(5,8) & (-20,0)};
(20,7)*{};(10,-5)*{} **\crv{(20,0)};
(23,0)*{C_0};
(-10,10)*{J_0};
(45,0)*{\cdots};
(2,-8)*{L};
(70,10)*{};(70,-10)*{} **\crv{(50,0)};
(70,10)*{};(90,15)*{} **\dir{.};
(70,-10);(90,-15)*{} **\dir{.};
(90,15)*{};(90,-15)*{} **\crv{(110,20) & (130,0) & (110,-20)};
(70,7)*{};(70,-7)*{} **\crv{(65,0)};
(100,10)*{};(100,-10)*{} **\crv{(110,0)};
(95,0)*{};(110,0)*{} **\dir{-};
(108,3)*{C_T};
(60,10)*{J_T};
(72,-8)*{L}
\endxy
\end{equation}

We should also remark that for higher blow-ups, similar arguments work to disjoin Lagrangians from divisors, but it is hard to find divisors whose complements are as well-understood as $T^*S^2$. For example, a Lagrangian sphere in the homology class $E_1-E_2$ in $\DD_n$ ($n<8$) can be disjoined from the exceptional spheres $E_3,\ \ldots,\ E_n$. Blowing-down these $n-2$ exceptional spheres leaves us with $\DD_2$. If $L$ and $L'$ are Lagrangian spheres in the homology class $E_1-E_2$ in $\DD_n$ ($n<8$) then we can isotope them both into the complement of the exceptional spheres $E_3,\ \ldots,\ E_n$ and blow-down these spheres to points $e_3,\ \ldots,\ e_n$. A corollary of theorem \ref{bigthm} is that any two Lagrangian spheres in $\DD_2$ are smoothly isotopic and we can choose that smooth isotopy to avoid the $n-2$ points $e_3,\ \ldots,\ e_n$. The result is a smooth isotopy of $L$ and $L'$ in $\DD_n$. This is certainly a symplectic phenomenon, as one can always knot smoothly embedded spheres topologically (for example by connect-summing locally with a smoothly knotted $S^2\subset\RR^4$).

In outline:
\begin{itemize}
\item In section \ref{prelim} we will review some basic facts about symplectic Del Pezzo surfaces, their homology and Lagrangian submanifolds. We also clarify the statement of theorem \ref{bigthm} and reformulate it as theorem \ref{isothm}.
\item Section \ref{gwtheory} reviews the tools we need from Gromov-Witten theory for studying what happens to the divisors and their linear systems under $\omega$-compatible deformations of the complex structure.
\item In section \ref{singularpseudo}, we describe the behaviour of these families of pseudoholomorphic curves under arbitrary $\omega$-compatible deformations of the complex structure.
\item Section \ref{SFT} reviews the techniques required from symplectic field theory (SFT) for analysing the limits of such families under a particular form of deformation of $J$ called neck-stretching.
\item In section \ref{neckstretching}, we give details of the setup for neck-stretching in our case and prove some basic properties of the limit.
\item Sections \ref{prf} and \ref{prf2} analyse the SFT limits of our families of pseudoholomorphic curves under neck-stretching and deduce disjointness of certain pseudoholomorphic curves from $L$.
\item This is used in section \ref{isoprf} to deduce theorem \ref{isothm}.
\end{itemize}

\section{Preliminaries}\label{prelim}

\subsection{Del Pezzo surfaces and symplectic forms}

A Del Pezzo surface $X$ is a smooth complex variety whose anticanonical bundle $-K_X$ is \emph{ample}, that is the sections of some tensor power of $-K_X$ define an embedding of $X$ into a projective space. The restriction of the ambient Fubini-Study form is then a K\"{a}hler form, $\omega$, on $X$. Note that we are normalising the Fubini-Study form to give a line in $\PP{N}$ area 1. The following classification theorem is well-known (see \cite{Reid} for example)

\begin{thm}
A Del Pezzo surface is biholomorphic to one of:

\begin{itemize}
\item A smooth quadric surface $Q\subset\PP{3}$ or, equivalently, a product $\PP{1}\times\PP{1}$ (thinking of the $\PP{1}$ factors as rulings on the quadric surface),
\item A blow-up, $\DD_n$, of $\PP{2}$ at $n$ points in general position for $n<8$.
\end{itemize}
\end{thm}

For convenience, we recall the second homology groups of these surfaces:

\begin{itemize}
\item $H_2(Q,\ZZ)=\ZZ^2$ is generated by the classes of the two rulings $\alpha=[\PP{1}\times \{\cdot\}]$ and $\beta=[\{\cdot\}\times \PP{1}]$. The intersection pairing is given by $\alpha^2=0=\beta^2$ and $\alpha\cdot\beta=1$. The first Chern class is Poincar\'{e} dual to $2\alpha+2\beta$.

\item $H_2(\DD_n,\ZZ)=\ZZ^{n+1}$ is generated by the class $H$ of a line in $\PP{2}$ and the classes $\{E_i\}_{i=1}^n$ of the $n$ exceptional spheres. The intersection pairing is given by $H^2=1$, $E_i\cdot E_j=-\delta_{ij}$ and $H\cdot E_i=0$. The first Chern class is Poincar\'{e} dual to $3H-\sum_{i=1}^nE_i$.
\end{itemize}

By definition, some multiple of the first Chern class is represented by the K\"{a}hler form so these are monotone symplectic manifolds.

\begin{rmk}\label{anticanform}If $n<7$ then the anticanonical bundle is \emph{very ample} i.e. its sections already define an embedding into projective space. For such blow-ups of $\PP{2}$, the induced K\"{a}hler form $\omega$ lies in the cohomology class $3H-\sum_{i=1}^nE_i$. It is sometimes too inexplicit to be useful, so we also work with a form $\omega'$ obtained by performing \emph{symplectic blow-up} (see \cite{MS05}, section 7.1) in $n$ symplectically (and holomorphically) embedded balls of volume $1/2$ in $(\PP{2},3\omega_{FS})$ centred at $n$ points in general position. Recall that under symplectic blowing up, a ball of volume $\pi^2\lambda^4/2$ is replaced by a symplectic $-1$-sphere of area $\pi\lambda^2$, so the cohomology class of $\omega'$ is again $3H-\sum_{i=1}^nE_i$. Since $\omega'$ is K\"{a}hler for the complex structure of the Del Pezzo surface, by Moser's theorem it is symplectomorphic (indeed isotopic) to $\omega$. For this reason, we will sometimes blur the distinction between $\omega$ and $\omega'$, writing $\omega$ for both.

For $n\geq 7$, one can always rescale $\omega$ so that its cohomology class is $3H-\sum_{i=1}^nE_i$. We call this the ``anticanonical K\"{a}hler form'' in the case when the anticanonical bundle is only ample.
\end{rmk}

\subsection{Lagrangian spheres}

The next lemma describes the homology classes in Del Pezzo surfaces which contain Lagrangian 2-spheres.

\begin{lma}
If $L$ is a Lagrangian sphere in $Q$ then it represents one of the homology classes $\pm(\alpha-\beta)$. If $L$ is a Lagrangian sphere in $\DD_n$ then it either represents a \emph{binary class} of the form $E_i-E_j$, a \emph{ternary class} of the form $\pm(H-E_i-E_j-E_k)$ (if $n\geq 3$) or a \emph{senary class} of the form $\pm(2H-\sum_{k=1}^6E_{i_k})$ (if $n\geq 6$).
\end{lma}
\begin{proof}
The requirement that $L$ is a Lagrangian sphere means that $[L]^2=-2$ and $c_1([L])=0$. For $Q$, if $[L]=A\alpha+B\beta$ then $2AB=-2$ and so $A=-B=\pm 1$. For $\DD_n$, suppose $[L]=dH+\sum a_i E_i$. Then $c_1([L])=0$ implies that
\[3d+\sum a_i=0\]
\noindent which, coupled with $[L]^2=-2$, gives
\[\left(\frac{-\sum a_i}{3}\right)^2-\sum a_i^2 = -2\]
\noindent or (after some reworking)
\[(9-n)\sum a_i^2 + \sum_{j<k}\left(a_j-a_k\right)^2=18\]
\noindent Since each summand is positive, it is a matter of combinatorics to check the claim.
\end{proof}

\begin{dfn}\label{lagtypes}
We call a Lagrangian sphere \emph{binary}, \emph{ternary} or \emph{senary} according to the arity of its homology class as defined in the previous lemma.
\end{dfn}

Notice that there is an intrinsic distinction here; it is not merely a matter of the basis we have chosen for $H_2(X,\ZZ)$. For instance, in $\DD_3$ there are six binary classes ($\{E_i-E_j\}_{i\neq j}$) and one ternary class ($H-E_1-E_2-E_3$). Computing intersection numbers: $(E_i-E_j)\cdot(E_k-E_{\ell})=\pm 1$ (when $\{i,j\}\neq\{k,\ell\}$) while $(E_i-E_j)\cdot (H-E_1-E_2-E_3)=0$. The symplectomorphism group acts with two orbits on the Lagrangian classes in $H_2(\DD_3,\ZZ)$: a binary one and a ternary one.

\subsection{Birational relations} \label{bir}

Apart from the blow-down maps
\[\rho_n:\DD_n\rightarrow\PP{2}\]
\noindent there are also blow-down maps
\[\pi^{ij}_n:\DD_n\rightarrow Q\]
\noindent for $n\geq 2$. To see this, observe that $Q$ is a quadric surface in $\PP{3}$ and one can birationally project it from a point $p\in Q$ to a hypersurface $\PP{2}\subset\PP{3}$. This map, $\phi$, is defined away from $p$. It collapses the lines $\alpha_p$ and $\beta_p$ through $p$ to points $a=\alpha_p\cap\PP{2}$ and $b=\beta_p\cap\PP{2}$, and its image otherwise misses the line $E$ through $a$ and $b$.

The graph of $\phi$ inside $Q\times\PP{2}$ is therefore isomorphic to $\DD_2$:
\[\begin{CD}
Q @<{\pi_2}<< \DD_2 \\
@. @VV{\rho_2}V \\
@. \PP{2}
\end{CD}\]
If we write $\pi_2$ and $\rho_2$ for the projections of this graph to $Q$ and $\PP{2}$ respectively then $\rho_2$ collapses the two spheres in the graph which respectively project one-to-one onto $\alpha_p$ and $\beta_p$ in $Q$. Similarly, $\pi_2$ collapses the sphere in the graph which projects via $\rho_2$ onto $E$. Thus $\pi_2:\DD_2\rightarrow Q$ is a blow-down. The homology class of the exceptional sphere is $H-E_1-E_2$.

Let us introduce the notation $S_{ij}$ for the homology class $H-E_i-E_j$ in $H_2(\DD_n,\ZZ)$. The map $\widetilde{\pi_2}$ is defined by blowing up the graph of $\phi$:
\[\begin{CD}
Q_{n-1} @<{\widetilde{\pi_2}}<< \DD_n \\
@. @VV{\rho_n}V \\
@. \PP{2}
\end{CD}\]
\noindent where $Q_{n-1}$ indicates $Q$ blown-up in $n-1$ points. Finally $\pi^{ij}_n$ is defined by composing $\widetilde{\pi_2}$ with the blow-down map to $Q$. The indices $ij$ are to indicate that the exceptional spheres of $\pi^{ij}_n$ are taken to be $\left\{E_k\right\}_{k\neq i,j}$ and $S_{ij}=H-E_i-E_j$.

\begin{rmk}
There is a symplectomorphism from the anticanonical K\"{a}hler form on $\DD_n$ with the symplectic blow-up of the anticanonical form on $Q$ (compare with remark \ref{anticanform}).
\end{rmk}

We construct some Lagrangian spheres in the binary homology class $E_i-E_j$ of $\DD_n$:

\begin{dfn}
Let $\overline{\Delta}$ be the antidiagonal Lagrangian sphere in $Q=\PP{1}\times\PP{1}$, i.e. the graph of the antipodal antisymplectomorphism $\PP{1}\rightarrow\PP{1}$. If the symplectic balls used to perform the blow-up $\pi_n^{ij}$ of $Q$ in $n-1$ points are chosen disjoint from $\overline{\Delta}$ then it lifts to a Lagrangian sphere in $(\DD_n,\omega')$ and hence specifies a Lagrangian sphere $\widetilde{\Delta}_{ij}$ in $(\DD_n,\omega)$ since $\omega\cong\omega'$.
\end{dfn}

\subsection{Disjointness from divisors}

Theorem \ref{bigthm} is proved by isotoping Lagrangian spheres until they are disjoint from a set of divisors whose complement is symplectomorphic to a compact subset of $T^*S^2$. Figures \ref{divisordiagram} and \ref{divisordiagram2} describe the relevant systems of divisors. The diagrams are to be interpreted as a union of smooth divisors, one for each line in the diagram, such that the homology class of the divisor corresponding to a given line is the one by which the line is labelled.

\begin{thm}\label{isothm}
A binary Lagrangian sphere in the homology class $E_1-E_2$ in $\DD_n$ (for $n\leq 4$) can be isotoped off a configuration of smooth divisors as shown in Figure \ref{divisordiagram}. A ternary Lagrangian sphere in the homology class $H-E_1-E_2-E_3$ can be isotoped off a configuration of smooth divisors as shown in Figure \ref{divisordiagram2}.
\end{thm}

\begin{figure}
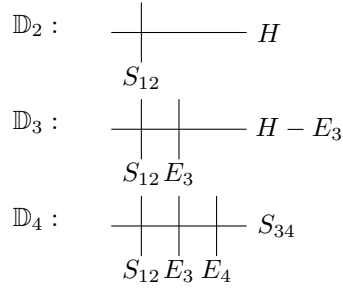

\begin{eqnarray*}
\DD_2: & & \xy (0,0)*{};(18,0)*{} **\dir{-};(21,0)*{H};(4,4)*{};(4,-4)*{} **\dir{-};(4,-6)*{S_{12}};\endxy \\
\DD_3: & & \xy (0,0)*{};(18,0)*{} **\dir{-};(25,0)*{H-E_3};(4,4)*{};(4,-4)*{} **\dir{-};(4,-6)*{S_{12}};(9,4)*{};(9,-4)*{} **\dir{-};(9,-6)*{E_3};\endxy \\
\DD_4: & & \xy (0,0)*{};(18,0)*{} **\dir{-};(22,0)*{S_{34}};(4,4)*{};(4,-4)*{} **\dir{-};(4,-6)*{S_{12}};(9,4)*{};(9,-4)*{} **\dir{-};(9,-6)*{E_3};(14,4)*{};(14,-4)*{} **\dir{-};(14,-6)*{E_4};\endxy\\
 \end{eqnarray*}
\caption{The configuration of smooth divisors from which a binary Lagrangian sphere (in the homology class $E_1-E_2$) can be made disjoint by a Lagrangian isotopy.}
\label{divisordiagram}
\end{figure}

\begin{figure}
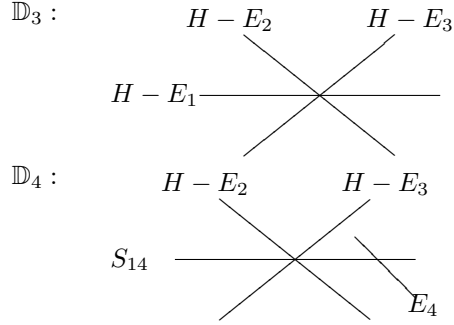

\begin{eqnarray*}
\DD_3: & & \xy(-16,-10)*{};(16,-10)*{} **\dir{-};(-10,-2)*{};(10,-18)*{} **\dir{-};(-10,-18)*{};(10,-2)*{} **\dir{-};(-22,-10)*{H-E_1};(-12,0)*{H-E_2};(12,0)*{H-E_3};\endxy \\
\DD_4: & & \xy(-16,-10)*{};(16,-10)*{} **\dir{-};(-10,-2)*{};(10,-18)*{} **\dir{-};(-10,-18)*{};(10,-2)*{} **\dir{-};(-22,-10)*{S_{14}};(-12,0)*{H-E_2};(12,0)*{H-E_3};(8,-7)*{};(16,-15)*{} **\dir{-};(17,-16)*{E_4}\endxy \end{eqnarray*}
\caption{The configuration of smooth divisors from which a ternary Lagrangian sphere (in the homology class $H-E_1-E_2-E_3$) can be made disjoint by a Lagrangian isotopy.}
\label{divisordiagram2}
\end{figure}

\begin{prp}
In each case the complement $U$ of the divisor contains a unique Lagrangian sphere up to isotopy.
\end{prp}
\begin{proof}
In each case we observe that $U$ is biholomorphic to an affine quadric surface:
\begin{itemize}
\item \textbf{Binary:} In figure \ref{divisordiagram}, the pictured divisors are total transforms of the diagonal in $\PP{1}\times\PP{1}$ under $\pi^{12}_n$.

\item \textbf{Ternary:} In the $\DD_3$ part of figure \ref{divisordiagram2}, the linear system of the divisor homologous to $H-E_1$ is a pencil with no basepoints and two nodal members (corresponding to the decompositions $S_{12}+E_2$ and $S_{13}+E_3$ of the homology class). The other divisors ($H-E_2$ and $H-E_3$) are sections of this pencil. Excising the pictured divisor gives a Lefschetz fibration over the disc with conic fibres and two nodal singularities. Using this Lefschetz fibration we can establish a biholomorphism with the affine quadric and its standard conic fibration with two singular fibres. For ternary spheres in $\DD_4$, the same trick works by looking at the linear system of $H-E_1$ and excising the singular fibre $S_{14}+E_4$ plus two sections (as depicted in the figure).
\end{itemize}
Now that we have established biholomorphism of $U$ with the affine quadric, notice that in each case we can give multiplicities to the components of the divisor so that it is linearly equivalent to the anticanonical class. Therefore we can assume that the restriction of $\omega$ to $U$ is the symplectic form associated to a plurisubharmonic function $\phi$ on $U$. It follows easily from \cite{SS05}, lemmas 5 and 6, that the symplectic completion of a sublevel set of $\phi$ is symplectomorphic to a complete (finite-type) Stein structure on $U$. Since all complete finite-type Stein structures on a given complex manifold are deformation equivalent, it follows that the symplectic completion of a sublevel set of $\phi$ is symplectomorphic to the affine quadric. The affine quadric is symplectomorphic to the total space of $T^*S^2$.

If there are two Lagrangian spheres $L_0$ and $L_1$ in $U$, they lie inside a sublevel set of $\phi$. Since the completion $Y$ of this sublevel set is symplectomorphic to $T^*S^2$, Hind's theorem gives us an isotopy $L_t$ in $Y$ between $L_0$ and $L_1$. Let $\psi_t$ be the negative Liouville flow on $Y$. For large $T$, $\psi_{-T}(L_t)$ is a Lagrangian isotopy between $\psi_{-T}(L_0)$ and $\psi_{-T}(L_1)$. Therefore the three-stage isotopy
\[\psi_t(L_0)|_{t=0}^{t=-T},\ \psi_{-T}(L_t)|_{t=0}^{t=1},\ \psi_t(L_1)|_{t=-T}^{t=0}\]
interpolates between $L_0$ and $L_1$ whilst remaining inside $U$.
\end{proof}

\begin{proof}[Proof of theorem \ref{bigthm}]
Theorem \ref{isothm} implies that a binary Lagrangian sphere can be isotoped into the complement of a specified divisor. The previous proposition tells us this complement contains a unique Lagrangian sphere up to Lagrangian isotopy in that space.
\end{proof}

\section{Gromov-Witten theory}\label{gwtheory}

This paper makes heavy use of the Gromov-Witten theory of genus 0 pseudoholomorphic curves as developed in \cite{MS04}. We now recall the basics of Gromov-Witten theory and the geometric theorems pertinent to dimension 4. We refer the reader to \cite{MS04} for proofs, where they are excellently presented.

\subsection{Basics}

\begin{dfn}
Let $(X,J)$ be an almost complex manifold. A $J$-holomorphic curve in $X$ is a smooth map $u:\Sigma\rightarrow X$ from a Riemann surface $(\Sigma,j)$ into $X$ whose derivative $Du$ satisfies
\[\dbar{J}u:=\frac{1}{2}\left(Du+J\circ Du\circ j\right)=0\]
\noindent $u$ is called \emph{multiply-covered} if it factors through a branched cover of Riemann surfaces and \emph{simple} otherwise. We concentrate on the case where $\Sigma$ has genus 0.
\end{dfn}

We think of $\dbar{J}$ as a section of the infinite-dimensional vector bundle
\[
\begin{CD}
\mathcal{E} @<<< \mathcal{E}_u=\Omega^{0,1}_J(u^*TX) \\
@VVV @. \\
\mC^{\infty}(S^2,X;A) @.
\end{CD}
\]
\noindent where $\mC^{\infty}(S^2,X;A)$ is the space of smooth maps $u$ from the sphere into $X$ for which $u_*[S^2]=A\in H_2(X;\ZZ)$.

\begin{dfn}[Moduli spaces of $J$-holomorphic spheres]
We write
\begin{itemize}
\item $\widetilde{\mM}(A,J)$ for the space $\dbar{J}^{-1}(0)$.
\item $\mM(A,J)$ for the quotient of $\widetilde{\mM}(A,J)$ by the reparametrisation action of $PSL(2,\CC)$ on genus 0 curves. This is called the moduli space of (unparametrised) $J$-holomorphic spheres in the class $A$.
\item $\widetilde{\mM}^*(A,J)$ for the space of simple curves in $\widetilde{\mM}(A,J)$.
\item $\mM^*(A,J)$ for the quotient of this space by reparametrisations. This is called the moduli space of (unparametrised) simple $J$-holomorphic spheres in the class $A$.
\end{itemize}
\end{dfn}

\begin{dfn}[Regular almost complex structures]
Let $d_u\dbar{J}$ denote the linearisation of $\dbar{J}$ at a point $u\in\mC^{\infty}(S^2,X;A)$. If $u$ is a $J$-holomorphic curve then $\dbar{J}(u)=0$ so that $T_{(u,0)}\mE$ can be naturally identified with $\mE_u\oplus T_u\mC^{\infty}(S^2,X;A)$. Write $\pr_{\mE_u}$ for the projection to the subspace $\mE_u$. A $J$-holomorphic curve $u$ is said to be \emph{regular} if $D_u\dbar{J}=\pr_{\mE_u}\circ d_u\dbar{J}$ is surjective as a map between suitable Sobolev completions of $T_u\mC^{\infty}(S^2,X;A)$ and $\mE_u$ and $J$ is called a \emph{regular almost complex structure} for the homology class $A$ if any $J$-holomorphic curve $u\in\widetilde{\mM}^*(A,J)$ is regular.
\end{dfn}

\begin{thm}[\cite{MS04}, theorem 3.1.5]
For a regular almost complex structure $J$, the space $\mM^*(A,J)$ is a manifold of dimension $2n+2\left<c_1(X),A\right>$, equipped with a natural smooth structure.
\end{thm}

In fact, we will mostly work with moduli spaces of \emph{stable maps}. These have the disadvantage that they do not possess a canonical manifold structure (smooth or topological), and are just Hausdorff topological spaces. However, they compensate by having good compactness properties. We will henceforth require that $J$ is compatible with a given symplectic form $\omega$ on $X$, that is $\omega(JX,JY)=\omega(X,Y)$ for all $X,Y$ and $\omega(X,JX)> 0$ for $X\neq 0$. In this setting, the moduli spaces of stable maps will be compact. Denote the space of $\omega$-compatible $J$ by $\mJ$.

Recall that a tree is a set $T$ of vertices connected by edges $E$ to form a graph with no cycles. A $k$-labelling of $T$ is an assignment of the integers $\{1,\ldots,k\}$ to the vertices of $T$, written $i\mapsto\alpha_i\in T$. We write $\dom(u)$ for the domain of a map $u$.

\begin{dfn}
A genus 0 $J$-holomorphic stable map with $k$ marked points modelled on a $k$-labelled tree $T$ is a collection of $J$-holomorphic spheres $u_{\alpha}:S^2\rightarrow X$, one for each vertex $\alpha\in T$, with:
\begin{itemize}
\item marked points $z_i\in\dom (u_{\alpha_i})$ for $i=1,\ldots,k$
\item nodal points $z_{\alpha\beta}\in\dom (u_{\alpha})$ for each oriented edge $\alpha\beta\in E$ such that $u_{\alpha}(z_{\alpha\beta})=u_{\beta}(z_{\beta\alpha})$.
\end{itemize}
\noindent We require all marked and nodal points to be distinct. The final requirement on this data is \emph{stability}: that the number of \emph{special points} (i.e. marked or nodal points) on $\dom (u_{\alpha})$ is at least 3 if $u_{\alpha}$ is a constant map. Such a constant component is called a \emph{ghost bubble}.

The stable map $(\mathbf{u},\mathbf{z})$ is said to represent a homology class $A\in H_2(X;\ZZ)$ if
\[A=\sum_{\alpha\in T} (u_{\alpha})_*[S^2].\]
\end{dfn}

To define a moduli space of stable maps we need a suitable notion of equivalence up to reparametrisation. We say $(\mathbf{u},\mathbf{z})$ and $(\mathbf{u'},\mathbf{z'})$ (modelled on labelled trees $T$ and $T'$ respectively) are equivalent if there is an isomorphism $f:T\rightarrow T'$ of trees and a reparametrisation $\phi_{\alpha}\in PSL(2,\CC)$ for each $\alpha\in T$ for which
\[u'_{f(\alpha)}\circ\phi_{\alpha}=u_{\alpha},\ z'_{f(\alpha)f(\beta)}=\phi_{\alpha}(z_{\alpha\beta}),\ z'_i=\phi_{\alpha_i}(z_i)\]
For each $k$-labelled tree $T$ this gives a reparametrisation group $G_T$ for the space of stable maps modelled on $T$ consisting of equivalences over automorphisms $f:T\rightarrow T$ for which $\alpha_i=f(\alpha_i)$ for all $i\in\{1,\ldots,k\}$.

\begin{dfn}
We define

\begin{itemize}
\item $\bM_{0,k}(X,A,J)$, the space of equivalence classes of genus 0 $J$-holomorphic stable maps with $k$ marked points representing the class $A$. We write $[\mathbf{u},\mathbf{z}]$ for the equivalence class of the stable map $(\mathbf{u},\mathbf{z})$.
\item $\widetilde{\mM}_{0,T}(X,A,J)$, the space of genus 0 $J$-holomorphic stable maps modelled on a labelled tree $T$.
\item $\mM_{0,T}(X,A,J)=\widetilde{\mM}_{0,T}(X,A,J)/G_T$, the space of equivalence classes of genus 0 $J$-holomorphic stable maps modelled on a labelled tree $T$.
\item $\widetilde{\mM}_{0,k}(X,A,J)$ respectively $\mM_{0,k}(X,A,J)$ to be the moduli space $\widetilde{\mM}_{0,T}(X,A,J)$ respectively $\mM_{0,T}(X,A,J)$ when $T$ is the $k$-labelled tree with one vertex.
\end{itemize}
\end{dfn}

We also introduce a notion of simplicity for stable maps, and suffix the notation for moduli spaces by $^*$ to denote restriction to simple stable maps.

\begin{dfn}
A stable map $(\mathbf{u},\mathbf{z})$ modelled on a tree $T$ is simple if every non-constant component is a simple $J$-holomorphic map and no two distinct vertices of $T$ give rise to non-constant maps with the same image.
\end{dfn}

We need to equip our moduli spaces with a topology. The \emph{Gromov topology}, defined in \cite{MS04}, section 5.6, will be the one relevant for Gromov-Witten theory. We recount those properties of the resulting moduli spaces which will be relevant for this paper:

\begin{thm}[\cite{MS04}, theorem 5.6.6]
Write $\bM_{0,k}(X,A,J)$ for the moduli space of genus 0 $J$-holomorphic stable curves with $k$ marked points equipped with the Gromov topology. Then $\bM_{0,k}(X,A,J)$ is:

\begin{itemize}
\item a separable Hausdorff space,
\item (Gromov's compactness theorem) compact,
\end{itemize}

\noindent and the evaluation map
\[\ev:\bM_{0,k}(X,A,J)\rightarrow X^k,\ \ev[\mathbf{u},\mathbf{z}]=(u_{\alpha_1}(z_1),\ldots,u_{\alpha_k}(z_k))\]
\noindent and the projection which forgets the $k$-th marked point
\[\bM_{0,k}(X,A,J)\rightarrow\bM_{0,k-1}(X,A,J)\]
\noindent are continuous maps.
\end{thm}

We note that
\[\bM_{0,k}(X,A,J) = \bigcup_T \mM_{0,T}(X,A,J)\]
\noindent and

\begin{eqnarray*}
\mM_{0,0}(X,A,J) & = & \mM(X,A,J) \\
\mM_{0,1}(X,A,J) & = & \mM(X,A,J)\times_{PSL(2,\CC)} S^2
\end{eqnarray*}

\noindent for $A\neq 0$. These equations also hold for the $\widetilde{\mM}$ and $\mM^*$ versions of the moduli spaces.

\subsection{Pseudocycles and transversality}

\begin{dfn}[Pseudocycles]
A $d$-dimensional pseudocycle in a manifold $X$ is a smooth map $f:V\rightarrow X$ from a smooth oriented $d$-manifold $V$ into $X$ such that:

\begin{itemize}
\item the image $f(V)$ has compact closure,
\item the \emph{limit set}
\[\Omega_f:=\bigcap_{\substack{K\subset V,\\K\mbox{ compact}}}\overline{f(V\setminus K)}\]
\noindent is of dimension at most $\dim(V)-2$.
\end{itemize}
Here a subset is \emph{of dimension at most $k$} if it is contained in the image of a map $f:A\rightarrow X$ where $A$ is a manifold of dimension $k$.
\end{dfn}

\begin{dfn}[Bordism of pseudocycles]
Two pseudocycles $f_1:V_1\rightarrow X$ and $f_2:V_2\rightarrow X$ are \emph{bordant} if there is a smooth $d+1$-manifold $W$ with boundary $\partial W=-V_1\cup V_2$ and a map $g:W\rightarrow X$ extending $f_0$ and $f_1$ such that $\Omega_g$ has dimension 2 less than $W$. Such data is called a \emph{bordism of pseudocycles}.
\end{dfn}

To state the next lemma which concerns intersections of pseudocycles we need a notion of transversality.

\begin{dfn}[Strong transversality]
Two pseudocycles $f_1:V_1\rightarrow X$ and $f_2:V_2\rightarrow X$ are \emph{strongly transverse} if their images are transverse wherever they intersect and neither limit set intersects the closure of the other pseudocycle, i.e.
\[\Omega_{f_i}\cap \overline{f_j(V_j)}=\emptyset\]
\end{dfn}

\begin{lma}[\cite{MS04}, lemma 6.5.5]\label{pseudoint}
Let $f_1:V_1\rightarrow X$ and $f_2:V_2\rightarrow X$ be pseudocycles of complementary dimension.

\begin{itemize}
\item There is a Baire set of diffeomorphisms $\phi$ of $X$ for which $\phi\circ f_1$ and $f_2$ are strongly transverse.
\item If they are strongly transverse then their intersection is finite. Define $f_1\cdot f_2$ to be the sum of the local intersection numbers of $f_1(V_1)$ and $f_2(V_2)$ at their intersection points.
\item The intersection number $f_1\cdot f_2$ is independent of the pseudocycles up to bordism.
\end{itemize}
\end{lma}

The pseudocycles we will consider are the \emph{Gromov-Witten pseudocycles} from the evaluation maps
\[\ev:\mM^*_{0,k}(X,A,J)\rightarrow X^k\]
\noindent for certain homology classes in a symplectic Del Pezzo surface $X=\DD_n$. In fact, all the classes considered will contain only simple curves, so the superscript $^*$ is superfluous. However, for general Gromov-Witten theory it is essential, so we leave it in here. The crucial theorem regarding these pseudocycles requires a finer notion of regularity for almost complex structures than those we have considered so far.

\begin{dfn}[GW-regularity]
Let $T$ be a $k$-labelled tree with a set of oriented edges $E$ and $A_{\alpha}$ a homology class in $H_2(X,\ZZ)$ for each vertex $\alpha\in T$. Write $A=\sum_{\alpha\in T}A_{\alpha}$. Set
\[Z(T)\subset(S^2)^{\#E}\times (S^2)^k\]
\noindent to be the set of all tuples $(z_{\alpha\beta},z_i)$ such that for every vertex $\alpha$, the points $z_{\alpha\beta}$ and $z_i$ are distinct for all $\beta$ and $i$ such that $\alpha_i=\alpha$.

Let $J$ be an almost complex structure on $X$ and consider the moduli space
\[\mM^*(\{A_{\alpha}\},T):=\prod_{\alpha\in T}\mM^*(A_{\alpha},J)\]
\noindent $J$ is \emph{regular for $T$ and $\{A_\alpha\}$} if

\begin{itemize}
\item each component $u_{\alpha}$ is a regular $J$-curve for every $\mathbf{u}\in\mM^*(\{A_{\alpha}\},T)$,
\item (\emph{edge transversality}) the map
\[\ev^{E}:\mM^*(\{A_{\alpha}\},T)\times Z(T)\rightarrow X^{\#E}\]
\noindent (sending $(\mathbf{u},\mathbf{z})$ to $u_{\alpha}(z_{\alpha\beta})$ for every oriented edge $\alpha\beta$) is transverse to the diagonal
\[\Delta^{E}:=\left\{z_{\alpha\beta}\in X^{\#E} : z_{\alpha\beta}=z_{\beta\alpha}\right\}\]
\end{itemize}

\noindent We say $J$ is \emph{GW-regular} for the class $A$ if it is regular for all $T$ and $\{A_{\alpha}\}$ such that $\sum_{\alpha\in T}A_{\alpha}=A$.
\end{dfn}

\begin{thm}[\cite{MS04}, theorem 6.6.1]\label{GW}
Let $(X,\omega)$ be a closed monotone symplectic manifold and $0\neq A\in H_2(X,\ZZ)$. If $J$ is GW-regular then the Gromov-Witten evaluation map
\[\ev:\mM^*_{0,k}(X,A,J)\rightarrow X^k\]
\noindent is a pseudocycle of dimension $2\dim(X)+2\left<c_1(X),A\right> +2k-6$ whose bordism class is independent of $J$ (amongst those which are GW-regular).
\end{thm}

Finally, we state a useful geometric criterion for GW-regularity.

\begin{lma}[\cite{MS04}, lemma 6.2.2]\label{usefulGWregular}
The edge evaluation map $\ev^E$ is transverse to $\Delta^E$ if and only if for every simple stable map $(\mathbf{u},\mathbf{z})\in\widetilde{\mM}^*_{0,T}(\{A_{\alpha}\},J)$, every edge $\alpha\beta\in E$ and every pair $v_{\alpha\beta}+v_{\beta\alpha}=0$ with $v_{\alpha\beta}\in T_{u_{\alpha}(z_{\alpha\beta})}X$ there are vectors
\[\xi_{\alpha}\in\ker D_{u_{\alpha}}\dbar{J},\ \zeta_{\alpha\beta}\in T_{z_{\alpha\beta}}S^2\]
\noindent such that
\[v_{\alpha\beta}=\xi_{\alpha}(z_{\alpha\beta})-\xi_{\beta}(z_{\alpha\beta})+du_{\alpha}(z_{\alpha\beta})\zeta_{\alpha\beta}-du_{\beta}(z_{\beta\alpha})\zeta_{\beta\alpha}\]
\end{lma}

This seemingly complicated lemma is actually a simple consequence of the formula
\[d\ev_{\alpha\beta}(\mathbf{u},\mathbf{z})(\mathbf{\xi},\mathbf{\zeta})=\xi_{\alpha}(z_{\alpha\beta})+du_{\alpha}(z_{\alpha\beta})\zeta_{\alpha\beta}\]
\noindent for the derivative of the map $\ev_{\alpha\beta}(\mathbf{u},\mathbf{z})=u_{\alpha}(z_{\alpha\beta})$.

\subsection{Geometry of pseudoholomorphic curves in 4D}

All the special features of pseudoholomorphic curves in dimension 4 stem from the following theorem of McDuff on their intersection properties. Here $(X,J)$ is an arbitrary almost complex 4-manifold.

\begin{thm}[Positivity of intersections \cite{MS04}, theorem 2.6.3]\label{posint}
Let $u_0$, $u_1$ be simple $J$-holomorphic curves in $X$ representing homology classes $A_0$ and $A_1$. If $u_0$ and $u_1$ have geometrically distinct images on every pair of open subsets of the domains then every intersection of the images contributes a positive integer to the homological intersection number $A_0\cdot A_1$. This number is 1 if and only if the intersections are transverse.
\end{thm}

\begin{thm}[Adjunction inequality \cite{MS04}, theorem 2.6.4]\label{adjform}
Let $u:\Sigma\rightarrow X$ be a simple $J$-holomorphic curve in $X$ and $A=u_*[S^2]\in H_2(X,\ZZ)$. Define
\[\delta(u):=\#\{(a,b)\in \Sigma\times\Sigma : a\neq b\mbox{ and } u(a)=u(b)\}\]
\noindent Then
\[\delta(u)\leq A\cdot A - \left<c_1(X),A\right>+\chi(\Sigma)\]
\end{thm}

\begin{thm}[Automatic transversality, \cite{HLS97} or \cite{MS04} lemma 3.3.3]\label{auttrans}
Let $u:S^2\rightarrow X$ be an embedded $J$-holomorphic sphere in an almost complex 4-manifold $X$. If $c_1(X)$ evaluates positively on $u$ then $u$ is regular.
\end{thm}

\section{Pseudoholomorphic curves in symplectic Del Pezzo surfaces}\label{singularpseudo}

The purpose of the next section is to prove that we can see the configurations of divisors specified in theorem \ref{isothm} and their linear systems even after perturbing the complex structure. Let $(X,\omega)$ be a symplectic Del Pezzo surface $\DD_n$ with its monotone blow-up form. We will assume $n\leq 7$. With this understood, we omit the target space $X$ from the notation for a moduli space of $J$-holomorphic maps. $J$ will denote an $\omega$-compatible almost complex structure on $(X,\omega)=\DD_n$. The following propositions are proved in sections \ref{area1}, \ref{area2} and \ref{area3} respectively.

\begin{prp}
For any $i\in\{1,\ldots,n\}$ and any $\omega$-compatible $J$ on $X$ there is a unique $J$-holomorphic stable curve $E_i(J)$ representing the homology class $E_i$. This curve is smooth, simple and embedded.
\end{prp}
This result is well-known (see for example, \cite{McDRat}, lemma 3.1) but we include a proof for completeness.
\begin{prp}\label{singfolexists}
For any $i\in\{1,\ldots,n\}$ and any $\omega$-compatible $J$ on $X$,
\[\ev:\bM_{0,1}(H-E_i,J)\rightarrow X\]
\noindent is a homeomorphism.
\end{prp}

\begin{prp}\label{Hclassexists}
For any $\omega$-compatible $J$ on $X$,
\[\ev_2:\bM_{0,2}(H,J)\rightarrow X\times X\]
\noindent is surjective.
\end{prp}

Let $\Xi(J)$ denote the union of the spheres $E_i(J)$. Let $\mJ_x$ denote the (non-empty) space of $J$ such that $x\not\in\Xi(J)$. This is non-empty because symplectomorphisms act transitively on points and $\mJ_x$ is non-empty for some $x$. Consider the evaluation map
\[\ev_1:\bM_{0,1}(H,J)\rightarrow X\]
and define the space $\bM_{0,0}(H,x,J)=\ev_1^{-1}(x)$ (which we can think of as unmarked stable $J$-curves in the class $H$ which pass through $x$). This final proposition is also proved in section \ref{area3}:

\begin{prp}\label{complextangent}
Denote by $\mathbb{P}_x^JX$ the space of $J$-complex lines in $T_xX$. For $J\in\mJ_x$, the map
\[\bM_{0,0}(H,x,J)\rightarrow \mathbb{P}_x^JX\]
\noindent sending a stable curve through $x$ to its complex tangent at the marked point $x$ is both well-defined (as $x\not\in\Xi(J)$) and a homeomorphism.
\end{prp}

\subsection{Area 1 classes: $E_i$, $S_{ij}$}\label{area1}

\begin{lma}\label{simpleenergy1}
A $J$-holomorphic sphere $u$ in $X$ with area $\int_{S^2}u^*\omega=1$ is simple and embedded.
\end{lma}
\begin{proof}
Since the area is minimal amongst non-zero spherical classes, $u$ cannot factor through a branched cover hence it is simple. Embeddedness will come from the adjunction formula:
\[\delta(u)\leq A\cdot A-\left<c_1(X),A\right> +2\]
\noindent where $A=u_*[S^2]\in H_2(X,\ZZ)$. Since $\left<c_1(X),A\right>=E(u)=1$, it remains to show that $A\cdot A<0$, for then the adjunction inequality becomes an equality $\delta(u)=0$, meaning that $u$ is an embedded sphere.

Suppose that $A=\alpha H+\sum_i\beta_i E_i$. Then

\begin{eqnarray*}
\left<c_1(X),A\right> & = & 3\alpha+\sum_i\beta_i \\
A\cdot A & = & \alpha^2-\sum_i\beta_i^2
\end{eqnarray*}

\noindent so

\begin{eqnarray*}
A\cdot A & = & \frac{\left(1-\sum_i\beta_i\right)^2}{9}-\sum_i\beta_i^2 \\
9A\cdot A & = & k+1-\sum_{i<j}\left(\beta_i-\beta_j\right)^2-\sum_i\left(\beta_i+1\right)^2-(8-k)\sum_i\beta_i^2
\end{eqnarray*}

We are required to show that this is strictly negative for all possible choices of $\beta_i$, which reduces to tedious case analysis. Once $\beta_i$ is large enough, the term $\sum_i\beta_i^2$ becomes large and negative, so there are very few cases that need to be checked. In order to obtain an integer class $A$, we also require that $\sum_i\beta_i\equiv 1\mod 3$, which is useful at a number of points in the case analysis. Note also that the proof fails for the case $n=8$, due to the existence of the class $3H-E_1-\cdots-E_8$ with square 1 and area 1.
\end{proof}

\begin{cor}\label{regularenergy1}
For any $J$, a $J$-sphere of area 1 is automatically regular.
\end{cor}
\begin{proof}
This follows directly from the automatic transversality lemma \ref{auttrans} once we know these spheres are embedded, since $c_1(X)$ evaluates to 1 on these homology classes.
\end{proof}

\begin{cor}\label{classifyclasses}
For any $J$, the only $J$-spheres of area 1 lie in the homology classes:

\begin{itemize}
\item $E_i$, $S_{ij}$ when $k<5$,
\item $E_i$, $S_{ij}$, $2H-\sum_{j=1}^5E_{i_j}$ when $k\geq 5$.
\end{itemize}
\end{cor}
\begin{proof}
By corollary \ref{regularenergy1} such a sphere $u$ is embedded and automatically regular. If $u$ represents a homology class $A$, then the adjunction inequality reads
\[0=\delta(u)\leq A\cdot A -c_1(A)+2\]
\noindent so
\[-1\leq A\cdot A\]
\noindent In the proof of lemma \ref{simpleenergy1} we saw that for such a sphere, $A\cdot A\leq -1$, and the only homology classes with $\left<[\omega],A\right>=1$ and $A\cdot A=-1$ are the ones listed in the statement of the corollary.
\end{proof}

\begin{lma}\label{breakingclasses}
Any genus 0 stable curve with one marked point and energy 1 is modelled on a single vertex tree.
\end{lma}
\begin{proof}
Certainly the tree for a stable curve of energy 1 can only have one non-constant component as 1 is the minimal energy for a pseudoholomorphic sphere. If it had a ghost bubble then it would have a ghost bubble corresponding to a leaf of the tree. The domain of this component would have at most two special points as there is only one marked point, contradicting stability of the curve. 
\end{proof}

\begin{lma}\label{existclasses}
For any $J$ there are unique $J$-holomorphic representatives of the area 1 classes $E_i$, $S_{ij}$ and (if possible) $2H-\sum_{j=1}^5E_{i_j}$.
\end{lma}
\begin{proof}
This is clear in the standard almost complex structure. Let $E$ denote one of these area 1 classes. The evaluation map
\[\mM^*_{0,1}(E,J_0)\rightarrow X\]
\noindent is a cycle in the homology class $E$ (the $\Omega$-limit set is empty). By lemma \ref{breakingclasses} the set of trees for which we must check GW-regularity is just the one-vertex tree and there are no edges so GW-regularity reduces to usual regularity of the almost complex structure. Therefore any $J$ is GW-regular by corollary \ref{regularenergy1}. Hence by theorem \ref{GW} the bordism class of the evaluation map (and hence the homology class of the image) is independent of $J$. Thus for any $J$ there is a $J$-holomorphic sphere in any of these three classes.

Uniqueness follows because these classes have homological self-intersection $-1$ and $J$-holomorphic curves intersect positively in dimension 4, so whenever two representatives intersect they must share a component. However, $J$-holomorphic representatives are smooth by lemma \ref{simpleenergy1} and therefore have a single component.
\end{proof}

\subsection{Area 2 classes: $H-E_i$}\label{area2}

\begin{lma}
A smooth $J$-holomorphic sphere $u$ in the homology class $H-E_i$ is simple and embedded.
\end{lma}
\begin{proof}
$H-E_i$ is a primitive class, hence any pseudoholomorphic representative is simple. In this case the adjunction formula gives
\begin{eqnarray*}
\delta(u) & \leq & (H-E_i)\cdot (H-E_i)-\left<c_1(X),H-E_i\right>+2 \\
&  & = 0-2+2=0
\end{eqnarray*}
\noindent and again $u$ must be embedded.
\end{proof}

\begin{cor}
For any $J$, a $J$-sphere in the homology class $H-E_i$ is automatically regular.
\end{cor}
\begin{proof}
As before, this follows directly from the automatic transversality lemma \ref{auttrans} once we know these spheres are embedded, since $c_1(X)$ evaluates to 2 on these homology classes.
\end{proof}

In particular, all smooth $H-E_i$-curves in the standard complex structure are regular. Let us examine the corresponding Gromov-Witten pseudocycle in a standard (integrable) complex structure obtained by blowing-up the standard structure on $\PP{2}$ at $n$ generic points. The image of the evaluation map:
\[\ev:\mM^*_{0,1}(H-E_i,J_0)\rightarrow X\]
\noindent is a pseudocycle in $X$ with complement the reducible complex codimension 1 subvariety consisting of the $2(n-1)$ exceptional curves in classes $E_j$ and $S_{ij}$ for all $j\neq i$.

\begin{lma}\label{breakinglines}
For any $J$, a genus 0 stable curve with one marked point in the moduli space $\bM_{0,1}(H-E_i,J)$ falls into one of three categories:
\begin{itemize}
\item A smooth $J$-sphere with a marked point,
\item A nodal curve with two smooth components, one marked, each of area 1,
\item A nodal curve with three components connected according to the tree $\xy (0,0)*{\dot{\circ}}; (5,0)*{\bullet} **\dir{-};(0,0)*{\dot{\circ}}; (-5,0)*{\bullet} **\dir{-};\endxy$ where the middle vertex corresponds to a marked ghost bubble and the outer spheres are each of area 1.
\end{itemize}
\end{lma}
\begin{proof}
There can be at most two non-constant components because the total area of any stable curve in the homology class $H-E_i$ is 2 and the minimal area of a pseudoholomorphic sphere is 1. There can be at most one ghost bubble because there is only one marked point to stabilise it.
\end{proof}

We have shown that all components of stable curves of the three types shown above are regular and also that the moduli spaces of nodal curves are nonempty. It remains to check transversality of the $\ev^{E}$ map to $\Delta^{E}$ to get GW-regularity.

\begin{itemize}
\item For the tree $\xy (0,0)*{\bullet};\endxy$, there are no edges, so GW-regularity reduces to usual regularity.
\item For the tree $\xy (0,0)*{\bullet}; (5,0)*{\bullet} **\dir{-};\endxy$, lemma \ref{usefulGWregular} reduces edge transversality to transversal intersection of the two energy 1 components (in this case the kernel of $D_{u_{\alpha}}\dbar{J}$ is trivial as the $u_{\alpha}$ curves are regular and index 0). Since these components are smooth $J$-spheres in the homology classes $E_j$ and $S_{ij}$ with $E_j\cdot S_{ij}=1$, by theorem \ref{posint} they intersect once transversely. Hence $J$ is regular for this tree.
\item For the tree $\xy (0,0)*{\dot{\circ}}; (5,0)*{\bullet} **\dir{-};(0,0)*{\dot{\circ}}; (-5,0)*{\bullet} **\dir{-};\endxy$ with a marked ghost bubble on the middle vertex $\gamma$, edge transversality is automatic from lemma \ref{usefulGWregular} since for both nodal points, $\ker D_{u_{\gamma}}\dbar{J}$ has (real) dimension 4.
\end{itemize}

The result of this is that any $J$ is GW-regular for the classes $H-E_i$, and theorem \ref{GW} implies that the bordism class of the pseudocycle
\[\ev:\mM^*_{0,1}(H-E_i,J_0)\rightarrow X\]
\noindent is independent of the almost complex structure we chose.

\begin{cor}\label{gromwit}
For any $\omega$-compatible almost complex structure $J$, there is a dense set of points in $X$ in the image of
\[\ev:\mM^*_{0,1}(H-E_i,J)\rightarrow X\]
\end{cor}
\begin{proof}
By lemma \ref{pseudoint} (1) there is a dense set of points which are strongly transverse to the pseudocycle $\ev$. The intersection number of such a point with the image of $\ev$ is 1, since that is the case for the standard complex structure and this intersection number is independent of the bordism class of the pseudocycle by lemma \ref{pseudoint} (3), which is independent of the almost complex structure by the remarks above.
\end{proof}

\begin{proof}[Proof of proposition \ref{singfolexists}]
It suffices to show that the evaluation map is bijective, for then it is a continuous bijection from a compact space $\bM_{0,1}(H-E_i,J)$ to a Hausdorff space $X$ and hence a homeomorphism.

\par\noindent \textbf{Surjectivity:} From corollary \ref{gromwit} there is a dense set of points in the image of
\[\ev:\mM^*_{0,1}(H-E_i,J)\rightarrow X\]
Suppose $x\in X\setminus\ev(\mM^*_{0,1}(H-E_i,J))$. Pick a sequence of points
\[x_i \in \ev(\mM^*_{0,1}(H-E_i,J))\]
\noindent tending to $x$ and a sequence of smooth marked $J$-curves $(u_i,z_i)$ with $u_i(z_i)=x_i$. This sequence has a Gromov convergent subsequence, by the Gromov compactness theorem, whose limit is a stable $J$-curve $[u,z]$ in $\bM_{0,1}(H-E_i,J)$ with $u(z)=x$.

\par\noindent \textbf{Injectivity:} Suppose there were a point $x\in X$ and distinct stable maps $[u,z]$ and $[u',z']\in\bM_{0,1}(H-E_i,X)$ for which $u(z)=x=u'(z')$.

If $u$ and $u'$ were both smooth curves then they would have to have the same image, or else they would intersect at $x$ and this intersection would contribute positively to their (zero) homological intersection by McDuff's theorem on positivity of intersections. If they had the same image, they would be reparametrisations of the same smooth curve, and hence correspond to the same stable map in the moduli space.

If $u$ were smooth and $u'$ were nodal then (forgetting marked points) $u'$ would be a stable $J$-curve corresponding to a splitting
\[H-E_i=S_{ij}+E_j,\ j\neq i\]
\noindent of homology classes. This follows directly from lemmas \ref{breakinglines}, \ref{classifyclasses} and the definition of the forgetful map $\bM_{0,1}(H-E_i,J)\rightarrow\bM_{0,0}(H-E_i,J)$ (see \cite{MS04}, 5.1.9). In particular, neither component of $u'$ can intersect the image of $u$ without contributing positively to either $(H-E_i)\cdot S_{ij}=0$ or $(H-E_i)\cdot E_j=0$. Hence there cannot be a point $x$ in both the image of $u$ and the image of $u'$.

Finally, consider the case when $u$ and $u'$ were both nodal. If they were stable curves corresponding to different splittings $S_{ij}+E_j$ and $S_{ik}+E_k$ then none of their components could possibly intersect contradicting their both passing through $x$. If they were to correspond to the same splitting of the homology class then their images would be geometrically indistinct. If the marked point were not mapped to the node then $u$ and $u'$ would clearly be reparametrisations of the same stable marked curve. If the marked point were mapped to the node then both stable maps would be modelled on the tree $\xy (0,0)*{\dot{\circ}}; (5,0)*{\bullet} **\dir{-};(0,0)*{\dot{\circ}}; (-5,0)*{\bullet} **\dir{-};\endxy$ with the marked point on the middle vertex, corresponding to the ghost bubble at the node. Any two such stable maps are equivalent by reparametrising the sphere corresponding to the middle vertex, so again $[u,z]=[u',z']$ as stable maps.
\end{proof}

\subsection{Area 3 classes: $H$}\label{area3}

As for curves in the class $H-E_i$, it is easily shown that a smooth curve in the class $H$ is simple, embedded and hence automatically regular. We will examine the Gromov-Witten pseudocycle
\[\ev_2:\mM^*_{0,2}(H,J)\rightarrow X\times X\]
\noindent for a standard complex structure obtained by blowing-up $n$ generic points on $\PP{2}$. If $(a,b)\in X\times X$ is a pair of distinct points neither of which lies on an exceptional curve $E_j$ then there is a unique line through them lifted from the line in $\PP{2}$. If $a=b\not\in E_j$ for any $j$ then there is a $\PP{1}$ of lines (some of which are singular curves: total transforms of lines through blow-up points) through $a=b$. Therefore the image of this standard pseudocycle has complement the reducible complex codimension 1 subvariety of pairs $(a,b)$ where one or both of $a$ or $b$ lies on a curve $E_j$.

\begin{lma}\label{pretty}
For any $J$, a genus 0 stable curve with two marked points in the moduli space $\bM_{0,2}(H,J)$ falls into one of 22 categories. We show these pictorially below. The symbol $\dot{\bullet}_q$ denotes a non-constant $J$-sphere with area $q$ and as many marked points as dots. The symbol $\dot{\circ}$ denotes a ghost bubble with as many marked points as dots.
\[\xy
(-15,0)*{\ddot{\bullet}_3};
(-15,-5)*{\bullet_3}; (-10,-5)*{\ddot{\circ}} **\dir{-};
(0,0)*{\ddot{\bullet}_2};(5,0)*{\bullet_1} **\dir{-};
(0,-5)*{\dot{\bullet}_2};(5,-5)*{\dot{\bullet}_1} **\dir{-};
(0,-10)*{\bullet_2};(5,-10)*{\ddot{\bullet}_1} **\dir{-};
(0,-15)*{\bullet_2};(5,-15)*{\bullet_1} **\dir{-};
(5,-15)*{\bullet_1};(10,-15)*{\ddot{\circ}} **\dir{-};
(5,-20)*{\bullet_2};(10,-20)*{\bullet_1} **\dir{-};
(0,-20)*{\ddot{\circ}};(5,-20)*{\bullet_2} **\dir{-};
(0,-25)*{\bullet_2};(5,-25)*{\dot{\circ}} **\dir{-};
(5,-25)*{\dot{\circ}};(10,-25)*{\dot{\bullet}_1} **\dir{-};
(0,-30)*{\dot{\bullet}_2};(5,-30)*{\dot{\circ}} **\dir{-};
(5,-30)*{\dot{\circ}};(10,-30)*{\bullet_1} **\dir{-};
(0,-35)*{\bullet_2};(5,-35)*{\ddot{\circ}} **\dir{-};
(5,-35)*{\ddot{\circ}};(10,-35)*{\bullet_1} **\dir{-};
(0,-40)*{\bullet_2};(5,-40)*{\dot{\circ}} **\dir{-};
(5,-40)*{\dot{\circ}};(10,-40)*{\dot{\circ}} **\dir{-};
(10,-40)*{\dot{\circ}};(15,-40)*{\bullet_1} **\dir{-};
(20,0)*{\ddot{\bullet}_1};(25,0)*{\bullet_1} **\dir{-};
(25,0)*{\bullet_1};(30,0)*{\bullet_1} **\dir{-};
(20,-5)*{\bullet_1};(25,-5)*{\ddot{\bullet}_1} **\dir{-};
(25,-5)*{\ddot{\bullet}_1};(30,-5)*{\bullet_1} **\dir{-};
(20,-10)*{\dot{\bullet}_1};(25,-10)*{\dot{\bullet}_1} **\dir{-};
(25,-10)*{\dot{\bullet}_1};(30,-10)*{\bullet_1} **\dir{-};
(20,-15)*{\dot{\bullet}_1};(25,-15)*{\bullet_1} **\dir{-};
(25,-15)*{\bullet_1};(30,-15)*{\dot{\bullet}_1} **\dir{-};
(20,-20)*{\dot{\bullet}_1};(25,-20)*{\dot{\circ}} **\dir{-};
(25,-20)*{\dot{\circ}};(30,-20)*{\bullet_1} **\dir{-};
(30,-20)*{\bullet_1};(35,-20)*{\bullet_1} **\dir{-};
(20,-25)*{\bullet_1};(25,-25)*{\dot{\circ}} **\dir{-};
(25,-25)*{\dot{\circ}};(30,-25)*{\dot{\bullet}_1} **\dir{-};
(30,-25)*{\dot{\bullet}_1};(35,-25)*{\bullet_1} **\dir{-};
(20,-30)*{\bullet_1};(25,-30)*{\dot{\circ}} **\dir{-};
(25,-30)*{\dot{\circ}};(30,-30)*{\bullet_1} **\dir{-};
(30,-30)*{\bullet_1};(35,-30)*{\dot{\bullet}_1} **\dir{-};
(20,-35)*{\ddot{\circ}};(25,-35)*{\bullet_1} **\dir{-};
(25,-35)*{\bullet_1};(30,-35)*{\bullet_1} **\dir{-};
(30,-35)*{\bullet_1};(35,-35)*{\bullet_1} **\dir{-};
(20,-40)*{\bullet_1};(25,-40)*{\ddot{\circ}} **\dir{-};
(25,-40)*{\ddot{\circ}};(30,-40)*{\bullet_1} **\dir{-};
(30,-40)*{\bullet_1};(35,-40)*{\bullet_1} **\dir{-};
(20,-45)*{\bullet_1};(25,-45)*{\dot{\circ}} **\dir{-};
(25,-45)*{\dot{\circ}};(30,-45)*{\dot{\circ}} **\dir{-};
(30,-45)*{\dot{\circ}};(35,-45)*{\bullet_1} **\dir{-};
(35,-45)*{\bullet_1};(40,-45)*{\bullet_1} **\dir{-};
(20,-50)*{\bullet_1};(25,-50)*{\dot{\circ}} **\dir{-};
(25,-50)*{\dot{\circ}};(30,-50)*{\bullet_1} **\dir{-};
(30,-50)*{\bullet_1};(35,-50)*{\dot{\circ}} **\dir{-};
(35,-50)*{\dot{\circ}};(40,-50)*{\bullet_1} **\dir{-};
\endxy\]
\end{lma}
\begin{proof}
The only cases that need to be ruled out are:
\[\xy
(-5,5)*{\ddot{\bullet}_1};(0,0)*{\circ} **\dir{-};
(5,5)*{\bullet_1};(0,0)*{\circ} **\dir{-};
(0,-5)*{\bullet_1};(0,0)*{\circ} **\dir{-};
(10,5)*{\dot{\bullet}_1};(15,0)*{\circ} **\dir{-};
(20,5)*{\dot{\bullet}_1};(15,0)*{\circ} **\dir{-};
(15,-5)*{\bullet_1};(15,0)*{\circ} **\dir{-};
\endxy\]
\noindent which would clearly result in three distinct area 1 curves intersecting at a single point. These classes must add up to $H$, and the only possibility (given corollary \ref{classifyclasses}) is that the three curves are $S_{ij}$, $E_i$ and $E_j$. Since $E_i$ and $E_j$ cannot intersect, these trivalent configurations are ruled out.
\end{proof}

The area 1 components are all understood: by corollary \ref{classifyclasses} they are of the form $E_j$, $S_{ij}$ or $2H-\sum_{j=1}^5E_j$. This latter case cannot occur, since the other components of a stable curve with total class $H$ would have to sum to $-H+\sum_{j=1}^5E_j$. This cannot be achieved by two area 1 curves, and there is no area 2 curve in this homology class since it would be somewhere injective (this is a primitive homology class) and have self-intersection -4, contradicting the adjunction inequality
\[0\leq A\cdot A-c_1(A)+2=-4\]
The area 2 components are therefore understood to be either $H-E_j$ or $H-S_{ij}=E_i+E_j$. This latter case is subsumed into the 3-component stable degenerations, as the unique curve in the class $E_i+E_j$ is the (disconnected) union of the exceptional spheres $E_i$ and $E_j$. One may check that all possible trees are GW-regular (as in the previous section) and so any $J$ is GW-regular for the class $H$. Theorem \ref{GW} implies that the bordism class of the pseudocycle
\[\ev_2:\mM^*_{0,2}(H,J)\rightarrow X\times X\]
\noindent is independent of $J$.

\begin{cor}\label{gwdensity}
For any $\omega$-compatible $J$ there is a dense set $\mathcal{D}$ of points in $X\times X$ in the image of $\ev_2$.
\end{cor}

\begin{proof}[Proof of proposition \ref{Hclassexists}]
Let $(x,y)\in X\times X$ and pick a sequence of points $(x_i,y_i)$ in $\mathcal{D}\subset X\times X$ (as defined in corollary \ref{gwdensity}) tending to $(x,y)$. By the corollary, we can choose a sequence of stable curves $u_i\in\ev_2^{-1}(x_i,y_i)$ and there is a subsequence of these which Gromov converges to a stable curve in $\bM_{0,2}(H,J)$ through $(x,y)$. This proves surjectivity.
\end{proof}

\begin{proof}[Proof of proposition \ref{complextangent}]
Fix $u$, a smooth $J$-curve homologous to $H$ but not passing through $x$. Consider the set $\upsilon=\ev_2^{-1}(\{x\}\times u)$ of stable curves through $x$ hitting points of $u$.

\begin{lma}
$\ev_2|_{\upsilon}:\upsilon\rightarrow \{x\}\times u$ is a bijection.
\end{lma}

\noindent Surjectivity comes immediately from proposition \ref{Hclassexists}. Injectivity will follow from positivity of intersections. Two curves homologous to $H$ intersecting at $x$ must intersect transversely as $x$ does not lie on a component $E_i(J)$ by assumption. If they were also to intersect at $p\in u$ they would have intersection number greater than $1$ by positivity of intersections, but $H\cdot H=1$.

Since $\ev_2$ is continuous and all spaces involved are compact and Hausdorff, $\upsilon$ is homeomorphic to the 2-sphere $\{x\}\times u$. We will now show that $\bM_{0,0}(H,x,J)$ is homeomorphic to the 2-sphere.

There are continuous forgetful maps
\begin{align*}
f_2:\bM_{0,2}(H,J)\rightarrow\bM_{0,0}(H,J) \\
f_1:\bM_{0,1}(H,J)\rightarrow\bM_{0,0}(H,J)
\end{align*}
The restriction of $f_1$ to $\bM_{0,0}(H,x,J)=\ev^{-1}_1(x)\subset\bM_{0,1}(H,J)$ is a homeomorphism. To see this, first note that $x\notin\Xi(J)$ implies that $x$ is not a node of any stable curve in $\bM_{0,0}(H,x,J)$. By lemma \ref{pretty} and the results of the earlier sections, if $u\in\bM_{0,0}(H,x,J)$ then all components of $u$ are simple, embedded curves. Furthermore $x$ occurs on precisely one of the components of $u$. Therefore the restriction of $f_1$ to $\bM_{0,0}(H,x,J)$ is a continuous bijection of compact Hausdorff spaces.

The restriction of $f_2$ to $\upsilon$ lands in $\bM_{0,0}(H,x,J)$
and is a bijection (as every sphere through $x$ will also hit $u$). By the same point-set topological reasoning it is a homeomorphism. Therefore $\bM_{0,0}(H,x,J)$ is homeomorphic to a 2-sphere.

Finally, we consider the map
\[\tau:\bM_{0,0}(H,x,J)\rightarrow\mathbb{P}_x^JX\]
\noindent which sends a stable curve through $x$ to its complex tangent at $x$. This is well-defined since by assumption $J\in \mJ_x$, so $x\not\in\Xi(J)$. By positivity of intersections, any two distinct stable curves through $x$ in the homology class $H$ intersect transversely at $x$, so $\tau$ is injective. But a continuous injection from a 2-sphere to a 2-sphere is a homeomorphism, so to prove the proposition it suffices to show $\tau$ is continuous.

\begin{lma}
$\tau$ is a continuous map.
\end{lma}

Since the Gromov topology is metrizable it suffices to prove that $\tau$ is sequentially continuous, i.e. that the complex tangent at $x$ of a Gromov-limit $v$ of a sequence $v_i$ of stable curves through $x$ is the limit of their complex tangents. Gromov convergence implies $\mC^{\infty}$-convergence of $v_i$ to $v$ on compact subsets away from the nodes and since $x\not\in\Xi(J)$ by assumption $x$ is always a smooth point of $v_i$ and of $v$. Hence the claim follows.
\end{proof}

\section{Symplectic field theory}\label{SFT}

\subsection{The setting}\label{setting}

\subsubsection{Contact-type hypersurfaces}\label{hypersurfaces}

\begin{dfn}
A \emph{contact-type hypersurface} in a symplectic $2n$-manifold $(X,\omega)$ is a codimension 1 submanifold $M$ for which there exists a collar neighbourhood $N\cong(-\epsilon,\epsilon)\times M$ and a \emph{Liouville vector field} $\eta$ defined on $N$ which is transverse to $\{0\}\times M$ and satisfies $\mathcal{L}_{\eta}\omega=\omega$.
\end{dfn}

To a contact-type hypersurface one associates:

\begin{itemize}
\item The \emph{contact 1-form} $\lambda=\iota_{\eta}\omega$,
\item The \emph{contact hyperplane distribution} $\zeta=\ker\lambda$,
\item The \emph{Reeb vector field} $R$ satisfying $\iota_R d\lambda=0$, $\lambda(R)=1$.
\end{itemize}

\begin{dfn}
Given a contact-type hypersurface $M$, the product $\RR\times M$ admits a symplectic structure $d(e^t\lambda)$ and the symplectic manifold $\Sympl(M)=(\RR\times M,d(e^t\lambda))$ is called the symplectisation of $M$. We also write $\Sympl_{\pm}(M)$ for the positive/negative parts $\RR_{\pm}\times M$ of $\Sympl(M)$.
\end{dfn}

Any contact-type hypersurface $M\subset X$ has a neighbourhood which is symplectomorphic to a neighbourhood of $\{0\}\times M$ in $\Sympl(M)$. This symplectomorphism is realised by the flow of the Liouville field (which is to be identified with $\partial_t$ in $\RR\times M$).

\begin{dfn}
Let $M$ be a contact-type hypersurface in $(X,\omega)$ with Liouville vector field $\eta$ and let $J$ be a complex structure on the bundle $\zeta\rightarrow M$. If $d\lambda(\cdot,J\cdot)$ is a nondegenerate bundle metric then we say $J$ is $d\lambda$-compatible. We can define an almost complex structure $\tilde{J}$ on $\Sympl(M)$ by requiring that:

\begin{itemize}
\item $\tilde{J}$ is $\RR$-invariant,
\item $\tilde{J}\partial_t= R$,
\item $\tilde{J}|_{\zeta}=J$.
\end{itemize}

\noindent If $J$ is $d\lambda$-compatible then $\tilde{J}$ is $d(e^t\lambda)$-compatible. We say that $\tilde{J}$ is a \emph{cylindrical almost complex structure} on $\Sympl(M)$. An $\omega$-compatible almost complex structure on $X$ which restricts to a cylindrical almost complex structure on a collar neighbourhood $N$ of $M$ is said to be \emph{adjusted to $M$ for the Liouville field $\eta$}.
\end{dfn}

\subsubsection{Contact-type boundary, cylindrical ends}\label{bdryends}

Now let $(X,\omega)$ be a compact symplectic manifold with a boundary component $M$. Suppose that there is:

\begin{itemize}
\item a collar neighbourhood $N\cong (-\epsilon,0]\times M$ of $M$,
\item a Liouville field $\eta$ defined in $(-\epsilon,0)\times M$,
\item a smooth extension of $\eta$ to $(-\epsilon,\epsilon')\times M$ (for some $\epsilon'$) which is transverse to $\{0\}\times M$.
\end{itemize}

\noindent Then we say that this boundary component of $X$ is of \emph{contact-type}. Furthermore, if $\eta$ points out of $X$ along $M$ then we say that $M$ is \emph{convex} and we say it is \emph{concave} otherwise. The collar neighbourhood $N$ of a convex (respectively concave) boundary component $M$ is symplectomorphic to a neighbourhood in $\Sympl_-(M)$ (respectively $\Sympl_+(M)$).

The Liouville field gives us all the data we formerly had on a contact-type hypersurface, namely a contact 1-form, a contact hyperplane distribution and a Reeb vector field.

\begin{dfn}
The \emph{symplectic completion} of a compact symplectic manifold $(X,\omega)$ with contact-type boundary $M$ comprising $b$ components $\bigcup_{i=1}^b M_i$ is the union

\[\overline{X}=X\cup \bigcup_{i=1}^b\Sympl_{\pm}(M_i)\]

\noindent where $\pm$ indicates $+$ if $M_i$ is convex and $-$ if $M_i$ is concave. The noncompact parts in the union are called the \emph{ends of $\overline{X}$}. The identifications for convex (respectively concave) ends are via symplectomorphisms defined by the Liouville flows $\eta_i$ on neighbourhoods $N_i\cong(-\epsilon,0]\times M_i$ (respectively $N_i\cong[0,\epsilon)\times M_i$) of $M_i$.
\end{dfn}

Thus $\overline{X}$ is a non-compact symplectic manifold and we write $\overline{\omega}$ for its symplectic form. $\overline{X}$ is diffeomorphic to the interior $\mathring{X}\subset X$ via a diffeomorphism $\phi$:

\begin{itemize}
\item $\phi$ is the identity on $X\setminus\bigcup_{i=1}^b N_i$,
\item $\phi$ sends each $(-\epsilon,\infty)\times M_i$ (or $(-\infty,\epsilon)\times M_i$) to $\mathring{N}_i\cong(-\epsilon,0)\times M_i$ (respectively $(0,\epsilon)\times M_i$) by a diffeomorphism of the form $(t,m)\mapsto (g_i(t),m)$,
\item the function $g_i:(-\epsilon,\infty)\rightarrow (-\epsilon,0)$ (respectively $g_i:(-\infty,\epsilon)\rightarrow(0,\epsilon)$):
\begin{itemize}
\item is monotone and concave,
\item coincides with $t\mapsto -\epsilon e^{-t}/2$ on $(0,\infty)$ (respectively $t\mapsto \epsilon e^t/2$ on $(-\infty,0)$)
\item coincides with the identity near $-\epsilon$ (respectively $+\epsilon$).
\end{itemize}
\end{itemize}

\begin{dfn}
A compact symplectic manifold with contact-type boundary is called a \emph{compact symplectic cobordism}. The completion of a symplectic cobordism with respect to some choice of Liouville fields is called a \emph{completed symplectic cobordism}.
\end{dfn}

Finally, we discuss almost complex structures.

\begin{dfn}
Let $(X,\omega)$ be a symplectic cobordism with boundary $M$. Let $\eta$ be a Liouville field defined in a collar neighbourhood $N$ of $M$, with contact form $\lambda$ and hyperplane distribution $\zeta$. An $\omega$-compatible almost complex structure on $X$ is \emph{adjusted to $M$ for the Liouville field $\eta$} if it is of the form $\tilde{J}$ on $N$ for a $d\lambda$-compatible $J$ on $\zeta$ (where we are using $\eta$ to identify $N$ with a subset of the symplectisation of $M$). We sometimes write \emph{$\eta$-adjusted} for brevity, or just \emph{adjusted} with a tacit choice of $\eta$.
\end{dfn}

\noindent An $\eta$-adjusted almost complex structure $J$ extends cylindrically to an $\overline{\omega}$-compatible almost complex structure $\overline{J}$ on the symplectic completion $\overline{X}$. We call this the \emph{cylindrical completion} of $(X,J)$.

\begin{lma}
$\phi_*\overline{J}$ is an $\omega$-compatible almost complex structure on $\mathring{X}$.
\end{lma}
\begin{proof}
If $v=a\partial_{\RR}+bR+\xi$ where $\partial_{\RR}$ is the Liouville direction, $R$ is the Reeb field and $\xi\in\zeta$, then

\begin{eqnarray*}
\overline{J}(v)&=&-b\partial_{\RR}+aR+J\xi \\
\phi_*\overline{J}(v)&=&-g'(g^{-1}(t))b\partial_{\RR}+a(g^{-1})'(t)R+J\xi
\end{eqnarray*}

\noindent so, as $\omega=d(e^t\lambda)=e^t(dt\wedge\lambda+d\lambda)$,

\[\omega(v,\phi_*\overline{J}v)=e^t\left(a^2(g^{-1})'(t)+b^2g'(g^{-1}(t))+d\lambda(\xi,J\xi)\right)\]

\noindent which is positive because $J$ is $d\lambda$-compatible and $g$ and $g^{-1}$ are both concave. Also,

\[\omega(\phi_*\overline{J}v,\phi_*\overline{J}w)=g'(g^{-1}(t))(g^{-1})'(t)(a_vb_w-a_w b_v)+\omega(\xi_v,\xi_w)\]

\noindent and the first term is just $(gg^{-1})'(t)=1$, therefore

\[\omega(\phi_*\overline{J}v,\phi_*\overline{J}w)=\omega(v,w)\]
\end{proof}

\subsubsection{Neck-stretching}

Given a contact-type hypersurface $M$ in a closed symplectic manifold $(X,\omega)$, one can cut along it to obtain a compact symplectic cobordism $X'$ with two boundary components $M_1$ and $M_2$. A Liouville field $\eta$ for $M$ restricts to two Liouville fields $\eta_1$ and $\eta_2$ on collar neighbourhoods of $M_1$ and $M_2$ respectively, so that these are boundary components of contact-type. Since $\eta$ is transverse to $M$, one of $M_1$, $M_2$ must be convex, the other concave. Without loss of generality, suppose $M_1$ is convex.

Let $J_1$ be an $\omega$-compatible almost complex structure on $X$ which is adjusted to $M$ for a Liouville field $\eta$. One forms a family $J_t$ of $\omega$-compatible almost complex structures on $X$ as follows:

\begin{itemize}
\item Let $\mathcal{F}_s$ be the flow of $\eta$.
\item Let $I_t=[-t-\epsilon,t+\epsilon]$ and define a diffeomorphism $\Phi_t:I_t\times M\rightarrow X$ by

\[\Phi_t(s,m)=\mathcal{F}_{\beta(s)}(m)\]

\noindent Here $\beta:I_t\rightarrow[-\epsilon,\epsilon]$ is a strictly monotonically increasing function satisfying $\beta(s)=s+t$ on $[-t-\epsilon,-t-\epsilon/2]$ and $\beta(s)=s-t$ on $[t+\epsilon/2,t+\epsilon]$.
\item Equip $I_t\times M$ with the symplectic form

\begin{equation}\label{phistaromega}
\omega_{\beta}=d\left(e^{\beta(s)}\iota_{\eta}\omega\right)=\Phi_t^*\omega
\end{equation}

\item Let $\tilde{J}_t$ be the $\eta$-invariant almost complex structure on $I_t\times M$ such that $\tilde{J}_t|_{\zeta}=J_{\zeta}$.
\item Glue the almost complex manifold $(X\setminus \Phi_t(I_t\times M),J)$ to $(I_t\times M,\tilde{J}_t)$ via $\Phi_t$. By equation \ref{phistaromega}, the result is symplectomorphic to $X$, but it has a new $\omega$-compatible almost complex structure, which we denote by $J_t$.
\end{itemize}

\begin{figure}[htb]
	\centering
		\includegraphics[height=5cm]{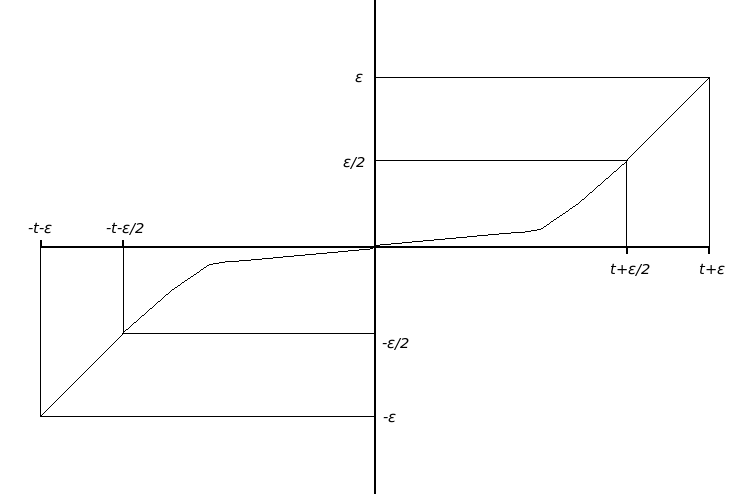}
	\caption{A valid function $\beta$ for defining $\omega_{\beta}$.}
	\label{fig:beta}
\end{figure}

This sequence $J_t$ of almost complex structures on $X$ is called a \emph{neck-stretch of $J_1$ along $M$}. We also consider the noncompact almost complex manifold $(X_{\infty},J_{\infty})$ which is the cylindrical completion of $(X',J_1)$.

\subsubsection{Reeb dynamics}

The closed orbits of the Reeb flow $\phi_t$ (generated by $R$) are important. Let $\tilde{N}_T$ denote the space of closed Reeb orbits in $M$ with period $T$.

\begin{dfn}
The Reeb flow $\phi_t$ is of Morse-Bott type if for every $T$ the space $\tilde{N}_T$ is a smooth closed submanifold of $M$ with $d\lambda|_{\tilde{N}_T}$ of locally constant rank and if the linearised return map $d\phi_T-\id$ is non-degenerate in the normal directions to $\tilde{N}_T$, i.e. its kernel is the tangent space to $\tilde{N}_T$.
\end{dfn}

The minimal period of a Reeb orbit $\gamma$ is the smallest period of any Reeb orbit $\gamma'$ with $\gamma(\RR)=\gamma'(\RR)$. We denote by $N_T$ the moduli orbifold $\tilde{N}_T/S^1$ of Reeb orbits; the orbifold points correspond to Reeb orbits with period $T/m$ covered $m$ times.

In the Morse-Bott case, we have the following lemma (\cite{Bou02}, lemma 3.1):

\begin{lma}\label{coords}
Let $M$ be a $2n-1$-manifold with contact form $\lambda$ and Reeb field $R$ of Morse-Bott type. Let $\gamma$ be a closed Reeb orbit of period $T$ with minimal period $\tau=T/k$. Suppose $N=\tilde{N}/S^1$ is the moduli orbifold of Reeb orbits containing $\gamma$. Then there is a tubular neighbourhood $V$ of $\gamma(\RR)$, a neighbourhood $U\subset S^1\times\RR^{2n-2}$ of $S^1\times\{0\}$ and a covering map $\psi:U\rightarrow V$ such that

\begin{itemize}
\item $V\cap \tilde{N}$ is invariant under the Reeb flow,
\item $\psi|_{S^1\times\{0\}}$ covers $\gamma(\RR)$ exactly $k$ times,
\item the $\phi$-preimage of a periodic orbit $\gamma'$ in $V\cap \tilde{N}$ is a union of circles $S^1\times\{a\}$ where $a\in\{0\}\times\RR^{\dim(N)}\subset\RR^{2n-2}$,
\item $\psi^*\lambda=f\lambda_0$, where $\lambda_0=d\theta+\sum_{i=1}^{n-1} x_idy_i$ is the standard contact form on $S^1\times\RR^{2n-2}$ and
\item $f$ is a positive smooth function $f:U\rightarrow\RR$ satisfying $f|_{S^1\times\{0\}\times\RR^{\dim(N)}}\equiv T$ and $D_{(\theta,0,a)}f=0$ for all $(0,a)\in\{0\}\times\RR^{\dim(N)}$.
\end{itemize}
\end{lma}

\subsection{Punctured finite-energy holomorphic curves}\label{puncgeneralities}

\subsubsection{Definitions}

Let $(X,\omega)$ be a compact symplectic cobordism with boundary $M=M_+\cup M_-$ (where $M_+$ denotes the union of convex components and $M_-$ the union of concave components) and let $J$ be an $\omega$-compatible almost complex structure on $X$ which is adjusted to $M$ for a choice of Liouville fields $\eta$. Let $(\overline{X},\overline{J})$ be the symplectic completion with its cylindrically-extended almost complex structure and denote by $E_+$ and $E_-$ the set of convex (respectively concave) ends of $\overline{X}$.

\begin{dfn}[Energy]
Let $(S,j)$ be a closed Riemann surface and

\[Z=\{z_1,\ldots,z_k\}\subset S\]

\noindent a set of punctures. A map $F:S\setminus Z\rightarrow \overline{X}$ has energy

\[E(F)=\int_{\overline{X}\setminus E_+\cup E_-}F^*\omega+E_{\lambda}(F)\]

\noindent where $E_{\lambda}$ is given by

\[\sup_{\phi_{\pm}\in\mC}\left(\int_{F^{-1}(E_+)}\left(\phi_+\circ a_+\right)da_+\wedge v_+^*\lambda_++\int_{F^{-1}(E_-)}\left(\phi_-\circ a_-\right)da_-\wedge v_-^*\lambda_-\right)\]

\noindent where $F|_{E_{\pm}}=(a_{\pm},v_{\pm})$ in coordinates on $\RR_{\pm}\times M_{\pm}$, $\lambda_{\pm}$ is the contact form on $M_{\pm}$ and $\mC$ consists of pairs of functions $(\phi_-,\phi_+)$ where $\phi_{\pm}:\RR_{\pm}\rightarrow \RR$ is such that

\[\int_0^{\infty}\phi_+(s)ds=\int_{-\infty}^0\phi_-(s)ds=1\]
\end{dfn}

\begin{dfn}[Punctured curves]
Let $(S,j)$ be a closed Riemann surface and $Z\subset S$ a set of punctures. A punctured finite-energy $\overline{J}$-holomorphic curve is a $(j,\overline{J})$-holomorphic map $F:\dot{S}:=S\setminus Z\rightarrow \overline{X}$ with finite energy. We write $(S,j,Z,F)$ for the data of a punctured finite-energy holomorphic curve.
\end{dfn}

We also define the oriented blow-up $S^Z$ of $S$ at the punctures to be the compactification of $\dot{S}$ obtained by replacing $z\in S$ with the circle $\Gamma_z$ of oriented real lines in $T_zS$ through $z$. This comes with a canonical circle action on $\Gamma_z$ because the complex structure on $T_zS$ allows us to define multiplication by $e^{it}$ which lifts to $\Gamma_z$. This will be relevant later when we examine the asymptotics of punctured curves.

\subsubsection{Asymptotics}

We begin by recalling the following important observation of Hofer \cite{Hof93}:

\begin{lma}\label{basicasymptotics}
Let $F=(a,v):\CC^*\rightarrow \RR\times M$ be a finite-energy $J$-holomorphic cylinder where $J$ is a cylindrical $d(e^t\lambda)$-compatible almost complex structure on $\Sympl(M)$ for some contact form $\lambda$ on $M$. Then there is a $T>0$, a Reeb orbit $\gamma$ and a sequence $R_k\rightarrow\infty$ such that

\begin{equation}\label{asymptote}
\lim_{k\rightarrow\infty}v(R_ke^{2\pi it})=\gamma(Tt)
\end{equation}

\noindent with convergence in $\mC^{\infty}(S^1,M)$.
\end{lma}

Suppose moreover that the Reeb flow on $(M,\lambda)$ is Morse-Bott. Lemma \ref{basicasymptotics} can be strengthened significantly in this setting. Suppose $F=(a,v)$ is a finite-energy cylinder and $\gamma$ is a Reeb orbit for which there is a sequence $R_k$ such that equation \ref{asymptote} holds.

\begin{prp}[\cite{HWZ96}, proposition 2.1]
Let $F=(a,v)$ be a map as in lemma \ref{basicasymptotics}. Let $\Gamma$ be the subspace of $\mC^{\infty}(S^1,M)$ consisting of $T$-periodic Reeb orbits in the Morse-Bott family containing $\gamma$. If $W$ is an $S^1$-invariant neighbourhood of $\Gamma$ then there is a constant $R_0$ such that for all $R\geq R_0$, $v(R,\cdot)\in W$.
\end{prp}

Now using the coordinates around $\gamma$ which we defined in lemma \ref{coords} and coordinates $(s,t)\in\RR\times S^1\cong\CC^*$, the Cauchy-Riemann equations for $F$ read:

\begin{eqnarray*}
0&=&\frac{\partial z}{\partial s}+J|_{\xi}\frac{\partial z}{\partial t}-\frac{1}{T^2}\left(a_t-a_s J|_{\xi}\right)\left(
\begin{array}{c}
0 \\
\nabla^2_{N^{\bot}}f
\end{array}
\right)z_{N^{\bot}} \\
0&=&\frac{\partial a}{\partial s}-\lambda(v_t) \\
0&=&\frac{\partial a}{\partial t}+\lambda(v_s) \\
\end{eqnarray*}

\noindent where $z$ is the projection of $v$ to $\RR^{2n-2}$, $J|_{\xi}$ is the matrix for the cylindrical almost complex structure on the contact hyperplanes $\xi$ and $z_{N^{\bot}}$ is the projection of $v$ onto the $2n-2-\dim(N)$-dimensional subspace $\left(\{0\}\times\RR^{\dim(N)}\right)^{\bot}$. Writing the first of these equations as

\[0=\frac{\partial z}{\partial t}+A(s,t)z_{N^{\bot}}\]

\noindent allows us to make the following definition:

\begin{dfn}
The \emph{asymptotic operator} of $F$ is the limit

\[A_{\infty}(t)=\lim_{s\rightarrow\infty}A(s,t):L^2(S^1,M)\rightarrow L^2(S^1,M)\]
\end{dfn}

To show this is well-defined requires some further asymptotic estimates (see \cite{HWZ96}, lemma 2.2). With this in hand, the following theorem can be shown.

\begin{thm}[\cite{HWZ96}, theorems 1.2-1.3 or \cite{Bou02}, chapter 3]\label{asymptotics}
Let $(M,\lambda)$ be a contact manifold with Morse-Bott Reeb flow. Let $J$ be a cylindrical adjusted almost complex structure on $\Sympl(M)$. Suppose that $F=(a,v):\CC^*\rightarrow \Sympl(M)$ is a finite-energy $J$-holomorphic curve and that $\gamma$ is a Reeb orbit for which there is a sequence $R_k$ such that equation \ref{asymptote} holds. Then

\[\lim_{R\rightarrow\infty} v(Re^{2\pi it})=\gamma(Tt)\]

\noindent as smooth maps $S^1\rightarrow M$.

Further, if $\RR\times\RR^{2n-2}$ are coordinates universally covering the coordinates from lemma \ref{coords} around $\gamma$ with respect to which $F$ is represented by

\begin{eqnarray*}
(a,v):[s_0,\infty)\times\RR & \rightarrow & \RR\times\RR\times\RR^{2n-2} \\
(a,v)(s,t) & = & (a(s,t),\theta(s,t),z(s,t))
\end{eqnarray*}

\noindent (for large $s_0$) then:
\begin{itemize}
\item There are constants $a_0,\theta_0\in\RR$ and $d>0$ such that

\begin{eqnarray*}
|\partial^{\beta}(a(s,t)-Ts-a_0)| & \leq & Ce^{-ds} \\
|\partial^{\beta}(\theta(s,t)-ks-\theta_0)| & \leq & Ce^{-ds} \\
\end{eqnarray*}

\noindent where $\beta$ is a multi-index, $C$ is a constant depending on the multi-index and $k$ is the number of times the orbit of period $T$ wraps the simple orbit of period $\tau$ with the same image. Notice that $\theta(s,t+\tau)=\theta(s,t)+k\tau$.
\item If $z\not\equiv 0$ (in which case $F$ would be a cylinder on $\gamma$), the following asymptotic formula holds for $z$:

\begin{equation}\label{asympz}
z(s,t)=e^{\int_{s_0}^s\mu(\sigma)d\sigma}\left(e(t)+r(s,t)\right)
\end{equation}

\noindent where $r(s,t)$ tends to zero uniformly with all derivatives as $s$ tends to infinity, $\mu:[s_0,\infty)\rightarrow\RR$ is a smooth function which tends to a number $L<0$ in the limit $s\rightarrow\infty$ and $e(t)$ is a vector in $\RR^2$.
\end{itemize}

More explicitly, the number $L$ and the vector $e(t)$ are an eigenvalue and corresponding eigenfunction for the asymptotic operator $A_{\infty}$ on $L^2(S^1,\RR^2)$.
\end{thm}

Though we have been talking in this section about punctured cylinders in symplectisations, the theory carries through for punctured finite-energy curves in cylindrical completions of symplectic manifolds with contact-type boundary.

\begin{prp}[\cite{BEHWZ03}, proposition 6.2]\label{compactifyingasymptotics}
Let $X$ be a compact symplectic cobordism with boundary $M$ and $J$ an adjusted almost complex structure. Suppose the Reeb flow on $M$ is Morse-Bott. Let $(S,j,Z,F:\dot{S}\rightarrow\overline{X})$ be a punctured finite-energy $\overline{J}$-holomorphic curve in the cylindrical completion $(\overline{X},\overline{J})$. Then for every $z\in Z$, either $F$ extends continuously over $z$ to a holomorphic map or there is a neighbourhood $\CC^*\subset\dot{S}$ of the puncture $z$ which is asymptotic to a Reeb orbit $\gamma_z$ in the sense of theorem \ref{asymptotics} above. Consequently, embedding $\overline{X}$ as $\mathring{X}\subset X$ via the diffeomorphism $\phi$ from \ref{bdryends} above, the curve $\phi\circ F$ extends continuously to a map $S^Z\rightarrow X$, sending the circle $\Gamma_i$ compactifying the puncture $z$ to the asymptotic orbit $\gamma_z$. This map is equivariant with respect to the canonical actions of $S^1$ on $\Gamma_z$ and (via the Reeb flow) on $\gamma_z$.
\end{prp}

If $M$ consists of convex components $M^+$ and concave components $M^-$, denote by $Z^{\pm}$ the sets of punctures of a holomorphic curve $(S,j,Z,F)$ which are asymptotic to Reeb orbits on $M^{\pm}$ and by $s^{\pm}$ the size of $Z^{\pm}$.

\subsection{Moduli spaces}

\subsubsection{Outline}

In what follows, $(X,\omega)$ is a compact symplectic cobordism with boundary $M$ and $J$ is an adjusted almost complex structure. Let $S\setminus Z$ be a punctured genus zero Riemann surface and $\left\{\rho_z\right\}_{z\in Z}$ be a collection of Morse-Bott manifolds of unparametrised Reeb orbits in $M$ with representative orbits $\gamma_z\in\rho_z$. For each $x\in Z$ we pick a Reeb orbit $r_z$ to act as a basepoint of $\rho_z$. A punctured finite-energy curve $F:S\setminus Z\rightarrow X$ such that the puncture $z$ is asymptotic to an orbit $\gamma_z$ from $\rho_z$ defines a relative homology class $A\in H_2(X,\bigcup_{z\in Z}r_z)$ as soon as we choose a path in $\rho_z$ from $\gamma_z$ to $r_z$ and consider this cylinder (sitting inside the boundary $M$ of $X$) glued to the end of the curve $F$. In all our cases, $\rho_z$ will be simply-connected, so the relative homology class thus defined will be independent of the choice of path. Once we have specified a collection of Morse-Bott manifolds of Reeb orbits and their basepoints, we can talk about the group $H_2(X,\bigcup_{z\in Z}r_z)$ and for a fixed class $A\in H_2(X,\bigcup_{z\in Z}r_z)$ we can hope to write down a moduli problem for all punctured finite-energy curves with these asymptotics defining this element of the relative homology group. The signs and multiplicities of orbits are subsumed into the notation: if a puncture is asymptotic to an orbit $\gamma$ with multiplicity $k$ then for our purposes it is asymptotic to the orbit $\gamma'=k\gamma$ and we work with the moduli space $\rho'$ of $\gamma'$.

To such a relative homology class $A$ one can assign a \emph{relative first Chern class} $c_1^{\Phi}(A)$ once one has chosen trivialisations $\Phi$ of the contact hyperplane distributions near the orbits $\gamma_z$:

\begin{dfn}[\cite{Wen08}]
Let $S\setminus Z\rightarrow X$ be a smoothly immersed representative of $A\in H_2(X,\bigcup_{z\in Z}\gamma_z)$ and $E$ be the pullback of the tangent bundle to $S\setminus Z$. Fix a trivialisation $\Phi$ of the contact hyperplane distributions near the orbits $\gamma_z$. There is an induced (unitary) trivialisation $\Phi$ of $E$ over the cylindrical ends of $S\setminus Z$. If $E=L_1\oplus\ldots\oplus L_k$ as a sum of smooth complex line bundles then define $c_1^{\Phi}(A)$ to be the sum $\sum_{i=1}^k c_1^{\Phi}(L_i)$ where $c_1^{\Phi}(L_i)$ is the number of zeros of a generic section of $L_i$ which restricts to a constant non-zero section over the ends (relative to the trivialisation $\Phi$).
\end{dfn}

Given the data $\left(S\setminus Z, \rho_Z=\left\{\rho_z\right\}_{z\in Z}, A\in H_2(X,\bigcup_{z\in Z}r_z)\right)$ (where $Z$ is understood as the union of positive and negative punctures $Z=Z^+\cup Z^-$ where $\rho_z$ consists of Reeb orbits in the convex respectively concave ends of $X$) we introduce the moduli space
\[\mM^A_{s^-,s^+}\left(\rho_Z,J\right)\]
\noindent of $\overline{J}$-holomorphic maps $S\setminus Z\rightarrow\overline{X}$ representing the class $A$ with $|Z^{\pm}|=s^{\pm}$. The complex structure on $S\setminus Z$ is allowed to vary and to give the moduli space a topology one must consider local Teichm\"{u}ller slices in the space of complex structures on $S\setminus Z$ (see for example, \cite{Wen08}). We will ignore this technicality since it does not affect the proof of the transversality results - we can achieve transversality without perturbing the complex structure on $S\setminus Z$.

In outline, the moduli space $\mM^A_{s^-,s^+}\left(\rho_Z,J\right)$ is the zero locus of a Cauchy-Riemann operator between suitable Banach manifolds whose linearisation is Fredholm of index
\[(n-3)\chi(S\setminus Z)+2c_1^{\Phi}(A)+\sum_{z\in Z}(\pm)^z K(z)\]
\noindent where $(\pm)^z$ is the sign of the puncture $z$ and
\[K(z)=\mu(\rho_z)(\pm)^z\frac{1}{2}\dim(\rho_z)\]
\noindent Here $\mu$ is a generalised Conley-Zehnder index for degenerate asymptotics defined in \cite{Bou02}, chapter 5 (depending on $\Phi$).

We also have evaluation maps
\[\ev_z:\mM^A_{s^+,s^-}\left(\rho_Z,J\right)\rightarrow\rho_z\]
\noindent sending a puncture to its asymptotic Reeb orbit and
\[\ev_R=\prod_{z\in Z}\ev_z:\mM^A_{s^+,s^-}\left(\rho_Z,J\right)\rightarrow R_Z:=\prod_{z\in Z}\rho_z\]

\subsubsection{Analytic set-up}

This section follows \cite{Wen08}. If $(S,j)$ is a closed Riemann surface with a finite set $Z=Z^+\cup Z^+$ of (positive and negative) punctures then one can pick cylindrical coordinates $(s,t):I_{\pm}\times S^1\rightarrow\aA^{\pm}_z$ on punctured disc neighbourhoods $\aA^{\pm}_z$ of the punctures, where $I_+=[0,\infty)$ and $I_-=(\infty,0]$. Let $S^Z$ be the oriented blow-up of $S$ at $Z$ with compactifying circles $\{\Gamma_z\}_{z\in Z}$. Given a Hermitian rank $r$ vector bundle $E$ over $S\setminus Z$ which extends continuously to a smooth complex vector bundle over $\bigcup_{z\in Z}\Gamma_z$, an \emph{admissible trivialisation} of $E$ near the ends is a smooth unitary bundle isomorphism
\[\Phi:E|_{\aA^{\pm}_z}\rightarrow I_{\pm}\times S^1\times\CC^r\]
\noindent which extends continuously to a unitary trivialisation of the bundle over the compactifying circles.

\begin{dfn}
Let $E$ be a bundle over $S\setminus Z$ and $\delta_Z=\{\delta_z\}_{z\in Z}$ a set of small positive numbers (these will eventually be taken to be smaller than the smallest eigenvalue of the asymptotic operator of the given puncture). A section $\sigma:S\setminus Z\rightarrow E$ is of class $W^{k,p}_{\delta_Z}$ if for each end $\aA^{\pm}_z$ there is an admissible trivialisation of $E$ over $\aA^{\pm}_z$ in which the section
\[(s,t)\mapsto e^{\pm\delta_zs}F(s,t)\]
is in $W^{k,p}(I_{\pm}\times S^1,\CC^r)$.
\end{dfn}

We now define our Banach space of maps.

\begin{dfn}
Let $S$ be a closed Riemann surface and $Z\subset S$ a finite set of punctures. For each $z\in Z$ assign a Morse-Bott family $\rho_z$ of Reeb orbits of period $T_z$ in $M$ and a number $\delta_z$. A map $F:\dot{S}\rightarrow \overline{X}$ is of class $W^{k,p}_{\delta_Z}$ if $F$ is in $W^{k,p}_{loc}$ and for each $z\in Z^{\pm}$ there is an orbit $\gamma_z\in\rho_z$ such that in cylindrical coordinates $(s,t)$ on an annular neighbourhood $\aA_z^{\pm}$ of $z$ in $S$
\[F(s+s_0,t)=\exp_{\tilde{\gamma}(s,t)}h(s,t)\]
\noindent for $s$ large enough, where $s_0$ is a constant, $\tilde{\gamma}(s,t)=(T_zs,\gamma_z(T_zt))\in\RR_{\pm}\times M$ and $h\in W^{k,p}_{\delta_z}(\aA^{\pm}_z,T(\RR_{\pm}\times M))$.
\end{dfn}

When $kp>2$ the space of $W^{k,p}_{\delta_Z}$-maps $F:S\setminus Z\rightarrow\overline{X}$ is a Banach manifold $\mB^{k,p}_{\delta_Z}$ with tangent spaces
\[T_F\mB^{k,p}_{\delta_Z}=W^{k,p}_{\delta_Z}(F^*T\overline{X})\oplus V_Z\oplus X_Z\]
To understand the summands $V_Z\oplus X_Z$, pick cylindrical coordinates $(s,t)$ on each end and coordinates $\RR\times S^1\times\RR^{2n-2}$ near each asymptotic orbit $\gamma_z$ as in lemma \ref{coords}. Let $\kappa^z_1,\kappa^z_2\in\mC^{\infty}(\aA^{\pm}_z,F|_{\aA^{\pm}_z}^*T\overline{X})$ be the vector fields given in local coordinates on $\RR\times S^1\times\RR^{2n-2}$ by $(1,0,0)$ and $(0,1,0)$. Extend these fields by a cut-off function to the rest of $S\setminus Z$. The $2|Z|$-dimensional space they span is defined to be $V_Z$. The construction of $X_Z$ is similar, but the vector fields at a given end are taken to be the constant fields in $\RR^{2n-2}$ tangent to the Morse-Bott space parametrising Reeb orbits.  $X_Z$ can be identified with $T_{\gamma_Z}R_Z$ if $\gamma_Z$ is the point in $R_Z$ representing the Reeb orbits to which $F$ is asymptotic.

Note that at this point, we can fix some extra data to obtain a related space. Let $Z_c\subset Z$ be a subset of punctures whose asymptotic orbit we fix to be $\gamma_z\in\rho_z$. We call these the \emph{constrained} punctures. The corresponding Banach manifold of maps is written $\mB^{k,p}_{\delta_Z}(Z_c)$ and the $X_Z$-summand in its tangent spaces is defined similarly but where one uses only the tangent vectors to Morse-Bott spaces parametrising Reeb orbits associated to unconstrained punctures.

Let $\eta$ be a choice of Liouville field near $M$ with contact form $\lambda$ and Reeb vector field $R$. Let $\mJ^{\ell}$ denote the space of $\omega$-compatible, $\eta$-adjusted, $\mC^{\ell}$-differentiable almost complex structures on $X$. $T_J\mJ^{\ell}$ is the Banach space of $\mC^{\ell}$-sections $Y$ of the endomorphism bundle $\End(T\overline{X})$ such that $Y\overline{J}+\overline{J}Y=0$, $\overline{\omega}(Yv,w)+\overline{\omega}(v,Yw)=0$, $Y(\xi)\subset\xi$ (where $\xi$ is the contact distribution) and $Y(R)=Y(\eta)=0$. This space is written $\mC^{\ell}(\End(T\overline{X},J,\eta,\omega))$.

Over the product $\mJ^{\ell}\times\mB^{k,p}_{\delta_Z}$ define the Banach bundle $\mE$ whose fibre at $(J,F)$ is the space
\[W^{k-1,p}_{\delta_Z}(\Lambda^{0,1}_j\otimes_{(j,\overline{J})}F^*T\overline{X})\]
\noindent and let $\sigma$ be the section
\[\sigma(J,F)=dF+\overline{J}\circ dF\circ j\]
A zero $(J,F)$ of $\sigma$ is precisely a finite-energy punctured $\overline{J}$-holomorphic curve.

\begin{dfn}
We define the universal moduli space of simple finite-energy curves to be the space $\mM^*(A,S,Z,\rho_Z,\mJ^{\ell})$ of $(J,F)\in\sigma^{-1}(0)$ such that $F$ does not factor through a multiple cover.
\end{dfn}

We have the following basic transversality results. The proofs are closely modelled on \cite{MS04}, propositions 3.2.1 and 3.4.2. and can be found in section \ref{transproofs} below.

\begin{prp}\label{trans1}
Let $e_z$ denote the smallest eigenvalue of the asymptotic operator $A_z$ assigned to the puncture $z$. For $\delta_Z\in \prod_{z\in Z}(0,e_z)$ the universal moduli space is a separable $\mC^{\ell-k}$ Banach submanifold of $\mJ^{\ell}\times\mB^{k,p}_{\delta_Z}$ and the projection map to $\mJ^{\ell}$ is a $\mC^{\ell-k}$-smooth Fredholm map whose index is given in theorem \ref{trans4} below.
\end{prp}

The condition on $\delta_Z$ is necessary for the Fredholm theory to work and implies the stated index formula. Only the transversality aspects of this proposition are proved below; the underlying Fredholm theory is referenced to the work of Schwarz \cite{SchwarzThesis} (see also \cite{WendlThesis}).

\begin{prp}\label{trans2}
Every point of $R_Z:=\prod_{z\in Z}\rho_z$ is a regular value of the evaluation map
\[\ev_R:\mM^*(A,S,Z,\rho_Z,\mJ^{\ell})\rightarrow R_Z\]
\end{prp}

Let $\pi_{\mJ}$ denote the projection of $\mM^*(A,S,Z,\rho_Z,\mJ^{\infty})$ to $\mJ^{\infty}$ where $\mJ^{\infty}$ denotes the subset of smooth almost complex structures. Note that by elliptic regularity the corresponding universal moduli space consists of smooth curves. The above transversality results together with the Sard-Smale theorem will prove the following results (see section \ref{transproofs}).

\begin{thm}\label{trans4}
For $J$ in a Baire set $\mJ_{\mbox{reg}}\subset\mJ^{\infty}$ the space
\[\mM^A_{s^+,s^-}(\rho_Z,J)=\pi_{\mJ}^{-1}(J)\]
\noindent is a smooth manifold of dimension
\[(n-3)\chi(S\setminus Z)+2c_1^{\Phi}(A)+\sum_{z\in Z}(\pm)^z K(z)\]
\noindent where $(\pm)^z$ is the sign of the puncture $z$ and
\[K(z)=\mu(\rho_z)(\pm)^z\frac{1}{2}\dim(\rho_z)\]
\noindent Here $\mu$ is a generalised Conley-Zehnder index for degenerate asymptotics defined in \cite{Bou02}, chapter 5 (depending on $\Phi$).
\end{thm}

\begin{thm}\label{trans5}
Suppose all asymptotic orbits are simple and fix a submanifold $P\subset R_Z$. For $J$ in a Baire set $\mJ_{\mbox{reg}}^P\subset\mJ_{\mbox{reg}}$ the evaluation map
\[\ev_R:\mM^A_{s^+,s^-}(\rho_Z,J)\rightarrow R_Z\]
\noindent is transverse to $P$.
\end{thm}

\subsubsection{Proofs of transversality results}\label{transproofs}

The proofs are closely modelled on \cite{MS04}, propositions 3.2.1 and 3.4.2.

\begin{proof}[Proof of proposition \ref{trans1}]
To prove that the universal moduli space is a manifold it suffices (by the implicit function theorem for Banach manifolds) to show that the vertical differential
\[D\sigma(J,F):W^{k,p}_{\delta_Z}(F^*T\overline{X})\times\mC^{\ell}(\End(T\overline{X},J,\eta,\omega)\rightarrow W^{k-1,p}_{\delta_Z}(\Lambda^{0,1}_j\otimes_{(j,J)}F^*T\overline{X})\]
\noindent is surjective for generic $\delta_Z$ whenever $F$ is a simple finite-energy $J$-curve. Notice that we have omitted the finite-dimensional $V_Z$ and $X_Z$ factors, as these do not affect the argument. This differential is given by
\[D\sigma(J,F)(Y,y,\xi)=D_F\xi+Y\circ dF\circ j\]
\noindent where $D_F\xi=\zeta ds+\overline{J}(F)\zeta dt$ in local isothermal coordinates on $S\setminus Z$ with $\zeta=\nabla_{\partial_sF}\xi+\overline{J}\nabla_{\partial_t}\xi+(\nabla_{\xi}\overline{J})\partial_tF$. Here $\nabla$ is the $\omega(-,\overline{J}-)$ Levi-Civita connection. By \cite{Bou02}, proposition 5.2, the operator $D_F$ is Fredholm and hence its image is closed. To prove surjectivity it suffices to prove that the image of $D_F$ is dense if $F$ is a simple $\overline{J}$-holomorphic curve.

\textbf{Case }$\mathbf{k=1}$: If the image of $D_F$ were not dense then by the Hahn-Banach theorem there would be a non-zero $\beta\in L^q_{-\delta_Z}(\Lambda^{0,1}_j\otimes_{(j,J)}F^*T\overline{X})$ with $q^{-1}+p^{-1}=1$ and
\begin{eqnarray}
\label{inteq1}\int_{S\setminus Z}\left<\beta,D_F\xi\right>d\vol & = & 0 \\ 
\label{inteq2}\int_{S\setminus Z}\left<\beta,Y\circ dF\circ j\right>d\vol & = & 0
\end{eqnarray}
The first of these three equations implies (via elliptic regularity) that $\beta$ is of Sobolev class $W^{1,p}$ and that $D^*\beta=0$. Aronszajn's unique continuation theorem means that $\beta$ can therefore only vanish on a discrete set of points.

It follows from \cite{Sie07}, corollaries 2.5 and 2.6 that a holomorphic curve which is simple has a finite number of non-injective points. We will show that $\beta$ vanishes on any non-injective point and hence vanishes identically. Let $x_0\in S\setminus Z$ be an injective point of $F$.

\begin{lma}
$\beta$ vanishes at $x_0$.
\end{lma}

If not, pick $Y_0\in\End(T_{F(x_0)}\overline{X},J_{F(x_0)},\eta_{F(x_0)},\omega_{F(x_0)})$ such that
\[\left<\beta_{x_0},Y_0\circ dF(x_0)\circ j(z_0)\right> >0\]
\noindent and let $Y\in\mC^{\ell}(\End(T\overline{X},\overline{J},\eta,\omega)$ be such that $Y(F(x_0))=Y_0$. This ensures that $\left<\beta,Y\circ dF\circ j\right>$ is positive in some neighbourhood of $x_0$. Since $x_0$ is an injective point this contains the inverse image under $F$ of a small ball $U_0$ centred at $F(x_0)$. If $C:\overline{X}\rightarrow[0,1]$ is a cut-off function supported in $U_0$ with $C(F(x_0))=1$ then
\[\left<\beta,CY\circ dF\circ j\right> \geq 0\]
\noindent and is positive at $x_0$ so equation \ref{inteq2} cannot hold for $CY$ so $\beta(x_0)=0$.

\textbf{Case }$\mathbf{k>1}$: Follows by elliptic regularity.
\end{proof}

Before proving proposition \ref{trans2} we prove the following lemma.

\begin{lma}\label{trans3}
Let $p>2$, $J\in\mJ^{\ell}$, $w\in S\setminus Z$ and $F:S\setminus Z\rightarrow \overline{X}$ be a simple punctured finite-energy $\overline{J}$-holomorphic curve. For every $\epsilon>0$ and tangent vector $v_Z\in T_{\ev_R(J,F)}R_Z$ there is a vector field $\xi\in T_{(J,F)}\mB^{\ell,p}_{\delta_Z}$ and a section $Y\in\mC^{\ell}(\End(T\overline{X},J,\eta,\omega)$ such that
\[D_F\xi+Y(F)dF\circ j=0,\ \mbox{supp}(Y)\subset B_{\epsilon}(u(w))\]
\noindent and the $X_Z$-component of $\xi$ is $v_Z$.
\end{lma}

\begin{proof}[Proof of lemma \ref{trans3}]
Let $\pi_X$ denote the projection of $T_F\mB^{k,p}_{\delta_Z}$ to $X_Z$. Given a vector field $\xi'\in T_F\mB^{k,p}_{\delta_Z}$ with $\pi_X(\xi')=v_Z$ we will show that there is a $\xi''\in T_F\mB^{k,p}_{\delta_Z}$ with $\pi_X(\xi'')=0$ and a $Y\in\mC^{\ell}(\End(T\overline{X},J,\eta,\omega))$ supported in $B_{\epsilon}(F(w))$ such that
\[D_F(\xi'+\xi'')+Y\circ dF\circ j=0\]
It suffices to show that the map
\[(\xi,Y)\mapsto D_F\xi+Y\circ dF\circ j\]
\noindent is surjective for $Y$ supported in $B_{\epsilon}(F(w))$ and $\xi\in\pi_X^{-1}(0)$, for then we may take $(\xi'',Y)$ in the preimage of $-D_F\xi'$. Let $\mathcal{Z}_k$ denote the image of this operator for $\xi$ of regularity $W^{k,p}_{\delta_Z}$.

\begin{lma}
$\mathcal{Z}^k=W^{k-1,p}_{\delta_Z}(\Lambda^{0,1}_j\otimes_{(j,J)} F^*T\overline{X})$ for $k\leq \ell$.
\end{lma}

This will follow by induction from elliptic regularity once we have proved that $\mathcal{Z}^1=L^p_{\delta_Z}(\Lambda^{0,1}_j\otimes_{(j,J)} F^*T\overline{X})$. Since $D_F:T_F\mB^{k,p}_{\delta_Z}\rightarrow L^p_{\delta_Z}(\Lambda^{0,1}_j\otimes_{(j,J)} F^*T\overline{X})$ is Fredholm (see \cite{WendlThesis}, section 4.5), the image of $W^{k,p}_{\delta_Z}(F^*T\overline{X})\oplus V_Z\oplus 0$ under $D_F$ is closed and it suffices to show it is dense.

If it were not dense then by the Hahn-Banach theorem there would be a nontrivial $\beta\in L^q(\Lambda^{0,1}_j\otimes_{(j,J)} F^*T\overline{X})$ with $q^{-1}+p^{-1}=1$, annihilating the image of $D_F$. This means that
\begin{eqnarray}
\label{inteq4}\int_{S\setminus Z}\left<\beta,D_F\xi\right>d\vol & = & 0 \\ 
\label{inteq5}\int_{S\setminus Z}\left<\beta,Y\circ dF\circ j\right>d\vol & = & 0
\end{eqnarray}
It is possible to show as in the proof of proposition \ref{trans1} that unless $\beta=0$ at an injective point of the curve one can derive a contradiction by considering some $Y$ supported in $B_{\epsilon}(F(w))$ making the integral (\ref{inteq5}) positive. Hence $\beta$ vanishes on any injective point. There is an open subset of injective points, hence the following lemma shows that $\beta$ vanishes identically:

\begin{lma}[\cite{MS04}, lemma 3.4.7]
Let $p>2$, $J\in\mJ^{\ell}$ and $F:S\setminus Z\rightarrow \overline{X}$ be a simple punctured finite-energy $J$-holomorphic curve. If there is a $\beta\in L_{\delta_Z}^q(\Lambda^{0,1}_j\otimes_{(j,J)}F^*T\overline{X})$ with $q^{-1}+p^{-1}=1$ satifying
\[\pi_X\xi=0 \Rightarrow \int_{S\setminus Z}\left<\beta,D_F\xi\right>d\vol=0\]
\noindent for every $\xi\in T_F\mB^{k,p}_{\delta_Z}$ then $\beta\in W^{\ell,p}_{\mbox{loc}}(\Lambda^{0,1}_j\otimes_{(j,J)}F^*T\overline{X})$ and $D^*_F\xi=0$ on $S\setminus Z$ where $D^*_F$ is the formal adjoint of $D_F$. Moreover $\beta\equiv 0$ if it vanishes on a nonempty open set.
\end{lma}
\begin{proof}
The proof from lemma 3.4.7. of \cite{MS04} carries through because it is only a local regularity result.
\end{proof}
This finishes the proof of lemma \ref{trans3}.\end{proof}

\begin{proof}[Proof of proposition \ref{trans2}]
The tangent space at $(J,F)$ to $\mM^*(A,S,Z,\rho_Z,\mJ^{\ell})$ is the subspace of $(\xi,Y)\in \left(W^{\ell,p}_{\delta_Z}(F^*T\overline{X})\oplus V_Z\oplus X_Z\right)\times\left(\mC^{\ell}(\End(T\overline{X},J,\eta,\omega))\right)$ satisfying $D_F\xi+Y\circ dF\circ j=0$. The derivative of the evaluation map $\ev_R:\mM^*(A,S,Z,\rho_Z,\mJ)\rightarrow R_Z$ is the projection to $X_Z$ which is naturally identified with the tangent space at $\prod_{z\in Z}\gamma_z$ to $R_Z$.

For any tangent vector $v_Z\in T_{\ev_R(J,F)}R_Z$, lemma \ref{trans3} applied to each component $S_j\setminus Z_j$ of $S\setminus Z$ gives us a tangent field $(\xi_j,Y_j)$ to the universal moduli space of that subcurve such that $\sum_j\xi_j$ projects to $v_Z$ in $X_Z$. Choose the points $w_j\in S_j\setminus Z_j$ such that the balls $B_{\epsilon}(F(w_j))$ are pairwise disjoint and do not intersect $F(\Sigma_k)$ unless $k=j$. This choice is possible by simplicity of $F$.

Finally, let $\xi$ be in $W^{\ell,p}_{\delta_Z}(F^*T\overline{X})\oplus V_Z\oplus X_Z$ such that $\xi|_{S_j\setminus Z_j}=\xi_j$ and $Y=\sum_jY_j$. The pair $(\xi,Y)$ is then in the preimage $d\ev_R^{-1}(v_Z)$ and $d\ev_R$ is surjective.
\end{proof}

\begin{proof}[Proof of theorems \ref{trans4}, \ref{trans5}]
These theorems are standard applications of the Sard-Smale theorem given the above propositions. For full details of these arguments, compare with \cite{MS04}, theorem 3.1.5.(\textsc{ii}). The formula for the dimension of the smooth moduli space comes from a Fredholm index formula for $D_F$, see \cite{Bou02}, section 5.
\end{proof}

\subsection{Compactness}

Let $(X_{\infty},J_{\infty})$ be the non-compact almost complex manifold obtained by stretching the neck around some contact-type hypersurface $M\subset X$ and let $X^k_{\infty}$ be the union

\[X_{\infty}\cup\coprod_{i=1}^k \Sympl_i(M)\]

$\Sympl_i(M)$ is understood to be equipped with the same cylindrical almost complex structure as the ends of $X_{\infty}$. If $M$ is a separating hypersurface cutting $X$ into two compact pieces then we denote these by $W$ and $V$, where $W$ has convex boundary and $V$ has concave boundary.

Consider a collection of finite-energy punctured $J_{\infty}$-holomorphic curves $F_{\nu}:\Sigma_{\nu}\setminus Z_{\nu}\rightarrow \Sympl_{\nu}(M)$ (for $\nu\in\{1,\ldots,k\}$) and $F_0:\Sigma_0\setminus Z_0\rightarrow X_{\infty}$. We allow these to have sets of marked points $K_{\nu}$ and sets of special marked pairs $D_{\nu}=\{\overline{d}_i,\underline{d}_i\}$ for which $F_{\nu}(\overline{d}_i)=F_{\nu}(\underline{d}_i)$ (creating a node $d^{\nu}_i$). Let $\Sigma_{\nu}^{Z_{\nu}}$ denote the oriented blow-up of $\Sigma_{\nu}$ at $Z_{\nu}$ with compactifying circles $\Gamma^{\nu}_z$, $z\in Z_{\nu}$. Let $\Gamma^{\nu}_{\pm}$ denote the union of the compactifying circles at positive/negative punctures in the $\nu$-th curve.

Recall from section \ref{bdryends} that $X_{\infty}$ was diffeomorphic to $X\setminus M$ via a diffeomorphism $\phi$. Let $N\cong [-\epsilon,\epsilon]\times M$ be a closed Liouville collar of $M$ and use the Liouville flow to define a new diffeomorphism $\phi^0$ from $X_{\infty}$ to $X\setminus N$. Define $N_{\nu}\subset N$ for $\nu=1,\ldots,k$ to be the subset $[-\epsilon+\frac{2(\nu-1)\epsilon}{k},\epsilon,+\frac{2\nu\epsilon}{k}]\times M$ and notice that the symplectic completion of this is the symplectisation of $M$. Therefore there is a diffeomorphism $\phi^{\nu}$ identifying $\Sympl_{\nu}M$ with the interior of $N_{\nu}$. Therefore the space $X^k_{\infty}$ maps into $X$ via the union of these diffeomorphisms with image the complement of a collection of embedded copies of $M$. See figure \ref{sympldecompo}.

\begin{figure}
	\centering
		\includegraphics[height=7cm]{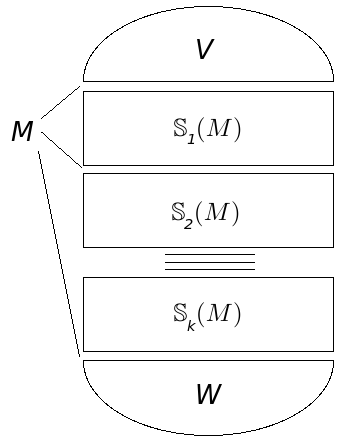}
	\caption{The decomposition of $X$ into $W$-, $V$- and symplectisation parts.}
	\label{sympldecompo}
\end{figure}

\begin{dfn}
The data $(F_{\nu},\Sigma_{\nu},Z_{\nu},K_{\nu},D_{\nu})$ defines a level $k$-holomorphic building in $X_{\infty}^k$ if there are a sequence $\{\Phi^{\nu}:\Gamma_+^{\nu}\rightarrow\Gamma_-^{\nu+1}\}_{\nu=1}^k$ of orientation-reversing diffeomorphisms (orthogonal on each boundary component) for which the compactifications of the maps $\phi^{\nu}\circ F_{\nu}$ glue to give a piecewise smooth map

\[\dot{F}:\Sigma^Z:=\bigcup_{\Phi^{\nu}}\Sigma_{\nu}^{Z_{\nu}}\rightarrow X\]

The \emph{genus} of a holomorphic building is the genus of the topological surface $\Sigma^Z$.
\end{dfn}

In the case where $M$ is separating, we write $F^W$ and $F^V$ for the $W$- and $V$-parts respectively, so $\Sigma_{0}=\Sigma^V\cup\Sigma^W$ and $F^W=F_{0}|_{\Sigma^W}$, $F^V=F_{0}|_{\Sigma^V}$.

\begin{dfn}
A marked level-$k$ holomorphic building is \emph{stable} if every constant component has at least three marked points and if there is no level $\nu$ for which all components of $F_{\nu}$ are unmarked Reeb cylinders.
\end{dfn}

For the relevant notions of equivalence and convergence for holomorphic buildings, we refer to the papers \cite{BEHWZ03} and \cite{CieMoh05} as the definitions are very involved. Instead, we remark that one can topologise the space of equivalence classes of stable level-$k$ genus $g$ buildings with $\mu$ marked points and $s^{\pm}$ positive/negative punctures with the topology of \emph{Gromov-Hofer convergence}. For us the most important property in the definition of Gromov-Hofer convergence is the following:

\begin{lma}
If $F_{\nu,j}:\Sigma_{\nu}\setminus Z_{\nu}\rightarrow X^k_{\infty}$ is a sequence of holomorphic buildings which Gromov-Hofer converge to a level-$k'$ building $F_{\nu}:\Sigma_{\nu}\setminus Z_{\nu}\rightarrow X^{k'}_{\infty}$ then the compactified curves $\dot{F}_j:\Sigma_j^{Z_j}:=\bigcup_{\Phi_j^{\nu}}\Sigma_{j,\nu}^{Z_{j,\nu}}\rightarrow X$ converge in $\mC^{\infty}_{\mbox{loc}}$ on compact subsets of the image of $X^k_{\infty}$ in $X$ and in $\mC^0$ everywhere.
\end{lma}

The reason for introducing these notions is the following set of compactness results, rephrased from the all-purpose theorem of \cite{BEHWZ03}:

\begin{thm}[\cite{BEHWZ03}, theorem 10.3]\label{SFTcompA}
Let $(X,\omega)$ be a closed symplectic manifold, $M$ a contact-type hypersurface and $J_1$ an $\omega$-compatible almost complex structure adjusted to some choice of Liouville fields near $M$. Let $J_t$ denote the family of almost complex structures obtained by neck-stretching $J_1$ along $M$ and let $u_j:S_j\rightarrow X$ be a sequence of $J_{t_j}$-holomorphic curves with $\omega$-energy bounded from above. Then there exists a subsequence $u_{j_{\ell}}$, a number $k$ and a level-$k$ holomorphic building $F$ in $X^k_{\infty}$ such that $u_{j_{\ell}}$ Gromov-Hofer converges to $F$.
\end{thm}

\begin{thm}[\cite{BEHWZ03}, theorem 10.2]\label{SFTcompB}
Let $(X,\omega)$ be a compact symplectic cobordism and $J$ an $\omega$-compatible almost complex structure adjusted to some choice of Liouville fields along the boundary. Then a sequence of $\overline{J}$-holomorphic buildings of level-$k$ with bounded energy in $\overline{X}$ has a Gromov-Hofer convergent subsequence whose limit is a level $k'$-holomorphic building for some $k'$.
\end{thm}

\section{Neck-stretching for Lagrangian spheres in Del Pezzo surfaces}\label{neckstretching}

In this section, $(X,\omega)$ will be a symplectic Del Pezzo surface $\DD_n$ and $L\subset X$ will be a Lagrangian sphere. To apply the machinery developed in the previous section, we extract from this data:

\begin{itemize}
\item A contact-type hypersurface $M$ enclosing a neighbourhood $W$ of $L$,
\item An almost complex structure $J_1$ on $X$ which is very explicitly given on $W$,
\item A family of almost complex structures $J_t$ arising from neck-stretching of $J_1$ along $M$.
\end{itemize}

We also need to understand specific properties of holomorphic curves in the symplectic completion $\overline{W}$.

\subsection{The neck-stretching data}\label{TS}

As $L$ is a Lagrangian 2-sphere, Weinstein's neighbourhood theorem guarantees the existence of a neighbourhood $\widetilde{W}$ of $L$ which is symplectomorphic to a neighbourhood of the zero-section in $T^*S^2$ with its canonical symplectic structure.

More explicitly, write $T^*S^2$ in coordinates:
\[T^*S^2=\{(u,v)\in\RR^3\times\RR^3 : |u|=1, u\cdot v=0\}\]
\noindent with canonical symplectic form $\omega_{\mbox{can}}=d\lambda_{\mbox{can}}$ where
\[\lambda_{\mbox{can}}=\sum_{j=1}^3 v_j du_j\]
\noindent and let $H$ be the Hamiltonian function
\[H(u,v)=\frac{1}{2}|v|^2\]
The Hamiltonian flow generated by $H$ is the cogeodesic flow on the round sphere and for $c>0$ the $c$-level set of $H$ is a contact-type hypersurface. The contact plane distribution is
\[\zeta_{(u,v)}=\left<(u\times v,0),(0,u\times v)\right>\]
\noindent and the Reeb vector field is $R=(v,0)$, the $\omega$-dual to $dH$. Write $M_c$ for $H^{-1}(c)$ and $W_c$ for $H^{-1}([0,c])$.

Suppose $\Psi:\widetilde{W}\rightarrow W_{2r}$ is the symplectomorphism given by Weinstein's neighbourhood theorem, taking $L$ to the zero-section. The hypersurface $M=\Psi^{-1}(M_r)$ is a contact-type hypersurface in $X$, bounding $W=\Psi^{-1}(W_r)$. The hypersurface $M$ is diffeomorphic to $\RR\mathbb{P}^3$.

We now write down an explicit $\omega_{\mbox{can}}$-compatible complex structure $I$ on $T^*S^2$ which we can restrict to $W_{2r}$, pull-back via $\Psi$ to $\widetilde{W}$ and extend arbitrarily but $\omega$-compatibly to the rest of $X$. The complex structure is obtained by identifying $T^*S^2$ with the affine quadric
\[\mQ=\{x_1^2+x_2^2+x_3^2=1\}\subset\CC^3\]
\noindent via the map $\varpi:T^*S^2\rightarrow\mQ$ given in coordinates by
\[x_j = u_j\cosh(|v|)+i v_j\sinh(|v|)/|v|\]
\noindent under which the Liouville form on $\mQ$ inherited from $\CC^3$ pulls back to (a positive multiple of) the canonical Liouville 1-form on $T^*S^2$:
\[\varpi^*\lambda_{\mQ}=\frac{\cosh(|v|)\sinh(|v|)}{|v|}\lambda_{\mbox{can}}\]
\noindent Hence the pull-back of the complex structure on the quadric preserves the contact distribution on $M$ and sends the canonical Liouville field $\eta=(0,v)$ to a rescaled Reeb field $(fv,0)$. Thus $I$ is $\omega$-positive and $\varpi$ identifies the Lagrangian zero-section with the real part of the quadric. We will interchangeably refer to these both as $L$.

\begin{dfn}
Given a Weinstein neighbourhood $\Psi$, a \emph{neck-stretching datum} for $L$ consists an extension of $\Psi^*I$ to an $\omega$-compatible almost complex structure $J_1$ on $X$. Note that such structures are $\eta$-adjusted on the cylindrical neck $\Psi^{-1}(W_{2r}\setminus W_r)$. Let $\mJ_1$ denote the space of neck-stretching data.
\end{dfn}

Apply the neck-stretching procedure from section \ref{bdryends} to $J_1\in\mJ_1$ along the hypersurface $M_r$. The result is a sequence $\{J_t\}_{t\in[1,\infty)}$ of $\omega$-compatible almost complex structures and we denote by $(X_{\infty},J_{\infty})$ the noncompact almost complex, symplectic manifold with cylindrical ends:
\[\overline{W}\cup\overline{V}\]
\noindent where $V$ is the closure of the complement of $W$ in $X$. We also write $X_{\infty}^k$ for the manifold
\[X_{\infty}\cup\coprod_{i=1}^k \Sympl_i(M)\]
\noindent where $\Sympl_i(M)$ is the symplectisation of $M$ (labelled by the integer $i$).

We will be interested in understanding the limits of $J_t$-holomorphic curves from the families constructed in section \ref{singularpseudo} as $t\rightarrow \infty$ for a neck-stretch $J_t$. To that end, let us first examine punctured finite-energy holomorphic curves in the $\overline{W}$-part of $X_{\infty}$, namely the affine quadric surface.

\subsection{Compactifying punctured curves}

\subsubsection{Asymptotics for the affine quadric $\mQ$}\label{asympoperator}

In the case of the affine quadric, the Reeb flow is actually a Hamiltonian circle action (see section \ref{TS}) so the manifold of periodic orbits of period $T$ fills the whole of $M$. The local model in theorem \ref{asymptotics} near a periodic orbit is just $S^1\times\RR^2$ with the standard contact form $\lambda=d\theta+x dy$. The Reeb orbits are then circles of constant $(x,y)$ and the linearised return map is the identity. This implies that the relevant self-adjoint operator $A$ in this case is just $-J_0\frac{d}{dt}$, whose eigenvalues $L$ are integer multiples of $2\pi$ and whose eigenfunctions are $f(t)=(\cos(Lt),\sin(Lt))$.

\subsubsection{Compactifying}\label{compactifying}

Recall that the symplectic completion $\overline{W}$ can be embedded in the original compact symplectic manifold $(W,\omega)$ with boundary $M$ by a diffeomorphism $\phi$ (see section \ref{bdryends}). Furthermore, $\phi_*\overline{J}$ is compatible with $\omega$. In the case of the affine quadric, one may take $W$ to be the closed $r$-sublevel set of the Hamiltonian $H$. Having a Hamiltonian circle action on this level set means we may perform a symplectic cut to obtain a closed symplectic manifold:
\[\wp(W):=\left(W\setminus M\right)\cup \left(M//S^1\right)\]
\noindent (see Lerman \cite{Ler95}). In fact, $\wp(W)$ is a monotone $S^2\times S^2$, and the compactification locus $M//S^1$ corresponds to the diagonal sphere $\Delta=\{(x,x): x\in S^2\}$. The symplectic form on the symplectic cut is proportional to the Poincar\'{e}-dual of $\Delta$.

Given a finite-energy punctured holomorphic curve $F:S\setminus Z\rightarrow \overline{W}$, the map $\phi\circ F$ extends continuously to a map $\dot{F}:S^Z\rightarrow W$. On $M$, the image of $F'$ is a collection of closed Reeb orbits and therefore $F'$ descends to a continuous map $\dot{F}:S\rightarrow\wp(W)$. Since $F$ is holomorphic, $\omega$ integrates positively over $\dot{F}(S\setminus Z)$ and the singular cycle represented by $\dot{F}$ satisfies:
\[[\Delta]\cdot[\dot{F}]>0\]
In terms of the basis $[S^2]\times\{0\},\{0\}\times [S^2]$ for $H_2(\wp(W),\ZZ)$, writing $[\dot{F}]=(a,b)$, this implies $a+b>0$.

\begin{rmk}\label{nosharedremark}
Let $F_1$ and $F_2$ be two finite-energy punctured holomorphic curves in $\overline{W}$ with $\Gamma_1$ and $\Gamma_2$ their sets of asymptotic Reeb orbits. Suppose that if $\gamma\in\Gamma_1$ then there is no $\gamma'$ in $\Gamma_2$ with the same image. Then the compactifications $\dot{F_1}$ and $\dot{F_2}$ only intersect in $\phi(\overline{W})$ and their intersections are precisely the images under $\phi$ of their intersections in $\overline{W}$.
\end{rmk}

\subsection{Relative first Chern class}\label{relc1}

An equivalent point of view of the symplectic cut of the affine quadric $\mQ$ is the projective quadric surface
\[\{z_0^2+z_1^2+z_2^2+z_3^2=0\}\subset\PP{3}\]
If $\Delta=\{z_0=0\}$ denotes a hyperplane section in $Q$ then the multiple divisor $2\Delta$ is in the anticanonical linear system. Hence the determinant line bundle $\det(TQ)$ has a holomorphic section $\sigma$ that is non-vanishing on $\mQ$.

In general, if $U\subset X$ is an open subset of a symplectic manifold which is symplectomorphic to a subset of $\mQ$ and $J$ is an $\omega$-compatible almost complex structure on $X$ agreeing with the restriction of the affine quadric complex structure on $U$ then $\det(TX)$ can be $J$-unitarily trivialised on the open set $U$ using this section $\sigma$. Thus with respect to this trivialisation $\Phi$, the relative first Chern class $c_1^{\Phi}(u)$ of a punctured holomorphic curve $u$ in $\mQ$ is 0.

One can define the first Chern class of a curve in the complement of the subset $U$ using this trivialisation near the boundary and trying to extend constant sections over the interior. If $f$ is a curve in $X$ which has a component $u$ in $U$ and $v$ in the complement of $U$ then
\[c_1^{\Phi}(u)+c_1^{\Phi}(v)=c_1(f)\]

In particular, suppose that $J_t$ is a sequence of $\omega$-compatible almost complex structures on $X$ arising from a neck-stretch and $f_t$ a Gromov-Hofer convergent sequence of $J_t$-holomorphic curves in a fixed homology class $C$. Let $f_{\infty}$ denote the $J_{\infty}$-holomorphic limit building with levels: $u$ landing in the completion of $U$, $s_k$ landing in the symplectisation at level $k$ and $v$ landing in the completion of the complement of $U$. Then, since the determinant line bundle can be unitarily trivialised on $U$ and on the symplectisation levels,
\[c_1^{\Phi}(v)=c_1(C).\]

\subsection{Punctured curves in the affine quadric: examples and properties}

The easiest way to obtain examples of punctured curves in $T^*S^2$ is to view it as the affine quadric $\mathcal{Q}\subset\CC^3$ as in section \ref{TS}. It is well-known that every point $p\in\mathcal{Q}$ lies on exactly two complex lines $\alpha_p$ and $\beta_p$ in $\CC^3$ with $\alpha_p,\beta_p\subset\mathcal{Q}$. Globally the fibration over $\mathcal{Q}$ whose fibre at $p$ is the two point set $\{\alpha_p,\beta_p\}$ is a 2-to-1 cover of $\mathcal{Q}$. Since $\pi_1(\mathcal{Q})=0$, the total space of this fibration has two components, corresponding to two distinct families of planes. Another way to see these families of planes is as follows: Given a choice of orientation on $L=\mbox{Re }\mathcal{Q}$ we can define $\alpha$-planes to be planes which intersect $L$ positively and $\beta$-planes to be planes which intersect $L$ negatively.

\begin{lma}
Thought of as punctured holomorphic planes, each $\alpha$- (respectively $\beta$-) plane is asymptotic to a single simple Reeb orbit on $M$. There is a unique $\alpha$- (respectively $\beta$-) plane asymptotic to each Reeb orbit.
\end{lma}
\begin{proof}
For the affine quadric, the projective compactification
\[Q=\{z_0^2+z_1^2+z_2^2+z_3^2=0\}\]
\noindent agrees with the symplectic cut, the projective compactification locus being the hyperplane section $\{z_0=0\}$. Affine lines projectively compactify to projective lines and therefore intersect the compactification locus in exactly one point transversely. This point is the representative of the unique asymptotic Reeb orbit, and every point on $\{z_0=0\}$ lies on a projective line of $\PP{3}$ contained in $Q$.
\end{proof}

\begin{rmk}\label{intersectionremark}
In fact, if $F$ is a finite-energy punctured holomorphic curve asymptotic to a Reeb orbit $\gamma$ of period $k\tau$ and minimal period $\tau$ then we can find the local intersection number of $\dot{F}$ and $\Delta$ by looking at the model case of an $\alpha$-plane asymptotic to $\gamma$, covered $k$-times by $\CC$ with branching over the origin. This projectively compactifies to a $k$-fold branched cover of a projective line, ramified at $0$ and $\infty$, and has local intersection number $k$ with the compactification locus $\Delta$.
\end{rmk}

\begin{lma}\label{mustbealphabeta}
Any finite-energy holomorphic plane asymptotic to a single simple Reeb orbit $\gamma$ must be an $\alpha$- or $\beta$-plane.
\end{lma}
\begin{proof}
If $\pi$ is neither an $\alpha$- nor a $\beta$-plane then, since these planes foliate $\mQ$, $\pi$ must intersect an $\alpha$-plane $A$ and a $\beta$-plane $B$ neither of which are asymptotic to $\gamma$. Furthermore, by positivity of intersections (which is a local property of holomorphic curves and therefore holds even in the current noncompact setting) this intersection contributes positively to the intersection of $\pi$ with $A$ and $B$. By remark \ref{nosharedremark} the compactifications satisfy
\[\dot{\pi}\cdot A > 0 \mbox{ , } \dot{\pi}\cdot B>0\]
\noindent In terms of the basis $[S^2]\times\{0\},\{0\}\times [S^2]$ for $H_2(\wp(W),\ZZ)$, writing $[\dot{\pi}]=(a,b)$, this implies $a+b\geq 2$. But $\pi$ is a plane asymptotic to a single simple Reeb orbit so the intersection number of $\dot{\pi}$ and $\Delta$ is 1 by remark \ref{intersectionremark}.
\end{proof}

\begin{lma}\label{cyl}
Let $F:S\setminus Z\rightarrow\mathcal{Q}$ be a punctured finite-energy holomorphic curve with a set $\Gamma$ of asymptotic Reeb orbits. If $A$ is an $\alpha$-plane and $B$ a $\beta$-plane, each asymptotic to some $\gamma'\notin\Gamma$ then either $F$ covers an $\alpha$- or $\beta$- plane or else it intersects both $A$ and $B$.
\end{lma}
\begin{proof}
Suppose $F$ does not cover an $\alpha$- or $\beta$- plane and (without loss of generality)
\[F(S\setminus Z)\cap A=\emptyset\]
Then $[\dot{F}]\cdot[\dot{A}]=0$. However the $\alpha$-planes foliate $\mathcal{Q}$ so the image of $F$ must intersect some $\alpha$-plane $A'$ whose asymptotic orbit is not contained in $\Gamma$. This intersection would be positive which would contradict $[\dot{F}]\cdot[\dot{A}]=0$.
\end{proof}

\subsection{Index formulas}

Richard Hind \cite{Hi04}, lemma 7, calculates the Conley-Zehnder indices of the Reeb orbits for this contact form on $M$ to be $2\cov(\gamma)$ where $\cov(\gamma)$ denotes the number of times the orbit $\gamma$ in question wraps around a simple Reeb orbit. We also know that the relative first Chern class of such a curve vanishes (see section \ref{relc1}). Therefore for genus 0, $s^+$-punctured curves in $W$, the index formula for the expected dimension of their moduli spaces reads:
\[\mathbb{E}\dim=2(s^+-1)+\sum_{i=1}^{s^+}2\cov(\gamma_i)\]
\noindent and for genus 0, $s^-$-punctured curves $F$ in $V$, the index formula becomes
\[\mathbb{E}\dim=2(s^--1+c_1^{\Phi}(F^*TV))-\sum_{i=1}^{s^-}2\cov(\gamma_i)\]

\section{Analysis of limit-buildings: binary case}\label{prf}

Let $L$ be a binary Lagrangian sphere in $\DD_n$, $n=2,3,4$, in the homology class $E_1-E_2$ (for notational simplicity). Let $\Psi$ a choice of Weinstein neighbourhood and $\mJ_1$ the space of neck-stretching data for $L$. In this section we will analyse the $J_{\infty}$-holomorphic buildings obtained as limits of sequences of stable $J_t$-curves from the families constructed in section \ref{singularpseudo}.

Let $E_k(J)$, $S_{k\ell}(J)$ denote the $J$-holomorphic exceptional spheres in the homology classes $E_k$ and $S_{k\ell}$. We will prove the following proposition:

\begin{prp}\label{grottyanalysis}
For generic neck-stretching data $J_1\in\mJ_1$ and large $t$,
\begin{itemize}
\item $E_k(J_t)\cap L=\emptyset$ if $k\neq 1,2$,
\item $S_{ij}(J_t)\cap L=\emptyset$ unless $|\{i,j\}\cap\{1,2\}|=1$,
\item If $i=3,4$ there is a smooth $J_t$-holomorphic curve homologous to $H-E_i$ disjoint from $L$,
\item There is a smooth $J_t$-holomorphic curve homologous to $H$ disjoint from $L$.
\end{itemize}
\end{prp}

Our limit analysis will actually prove this for $t=\infty$, but the nature of Gromov-Hofer convergence then implies the proposition.

\subsection{Exceptional spheres}

We postpone analysis of the exceptional classes $E_3$, $E_4$ and $S_{34}$ to the end of the section and first examine the limit of a convergent subsequence $E(J_{t_j})$ when $E$ is an exceptional class $E_1$, $E_2$ or $S_{12}$. Such a sequence exists by the SFT compactness theorem \ref{SFTcompA}. Let $E_{\infty}$ denote the limit building in $X^k_{\infty}$ for some $k$. Recall that its $V$- and $W$-parts are denoted $E_{\infty}^V$ and $E_{\infty}^W$ respectively.

\begin{lma}\label{propertiesexceptional}
$E_{\infty}^V$ is non-empty, simple and connected.
\end{lma}
\begin{proof}
$E_{\infty}^V$ is non-empty because the maximum principle for holomorphic curves forbids closed holomorphic curves in $W$. The discussion in section \ref{relc1} implies that the relative first Chern class
\[c_1^{\Phi}(E_{\infty}^V)=c_1(E)=1.\]
Any finite-energy punctured curve $v$ in $V$ has $c_1^{\Phi}(v)>0$ since one could glue it continuously to a collection of finite-energy planes in $W$ to obtain a singular cycle $C$ in $X$ which is symplectic away from $M$ so that $\omega(C)>0$. Since $[c_1(X)]=[\omega]$ and $c_1^{\Phi}(u)=0$ for a holomorphic curve $u$ in $W$, $c_1^{\Phi}(v)=c_1(C)=\omega(C)>0$. Therefore each component of $E_{\infty}^V$ contributes positively to $c_1^{\Phi}(E_{\infty}^V)=1$, but 1 is primitive amongst relative Chern classes. Therefore $E_{\infty}^V$ is connected. If $v'$ is a branched $k$-fold cover of $v$ then $c_1^{\Phi}(v')=kc_1^{\Phi}(v)$, so $E_{\infty}^V$ is simple.
\end{proof}

\begin{rmk}\label{regularityforexceptional}
By theorem \ref{trans4}, we can choose a neck-stretching datum $J_1$ such that $J_{\infty}|_V$ is regular for $E_{\infty}^V$. In fact, since there is a countable number of possible connected topologies for the domain of $E_{\infty}^V$ and a finite number of possible relative homology classes for each of those giving $c_1^{\Phi}=1$, we can take an intersection of the Baire sets for each Fredholm problem and what remains is a Baire set in $\mJ_1$ making all genus 0 connected $J_{\infty}|_V$-holomorphic punctured curves in $V$ with $c_1^{\Phi}=1$ regular.
\end{rmk}

In this generic case, the dimension formula for the moduli space of punctured finite-energy curves in $V$ containing $E_{\infty}^V$ becomes:

\begin{eqnarray*}
\dim(\mM(S,J_{\infty},E_{\infty}^V)) & = & -2+2s^{-}+2c_1^{\Phi}(E_{\infty}^V)-2\sum_{i=1}^{s^-}\cov(\gamma_i) \\
& = & 2\left(s^--\sum_{i=1}^{s^-}\cov(\gamma_i)\right) \\
& \leq & 0
\end{eqnarray*}

\noindent with equality if and only if $\cov(\gamma_i)=1$ for all $\gamma_i$, where $\{\gamma_i\}_{i=1}^{s^-}$ is the set of Reeb orbits to which $E_{\infty}^V$ is asymptotic at the punctures $\{z_i\}_{i=1}^{s^-}$. Thus, generically, the asymptotic Reeb orbits are simple.

\begin{rmk}\label{transvevenergy1}
This means that for generic neck-stretching data, all moduli spaces of $c_1^{\Phi}=1$, genus 0 punctured finite-energy curves in $V$ are zero-dimensional and all such curves have simple Reeb asymptotics. By theorem \ref{trans5}, one can also ensure that the puncture-evaluation maps from such moduli spaces to products of $\rho$ are transverse to all strata of the multi-diagonal. That is, if a $c_1^{\Phi}=1$, genus 0 curve has $k$ punctures, the map sending its moduli space to $\rho^k$ is transverse to all smooth strata of the subset $\{(\gamma_1,\ldots,\gamma_k):\gamma_i=\gamma_j\mbox{ for some }i,j\}$. In particular, for generic neck-stretching data, a $c_1^{\Phi}=1$, genus 0 punctured finite-energy curve in $V$ has \emph{distinct}, simple Reeb asymptotics.
\end{rmk}

\begin{lma}\label{reebcyl}
Any part of the limit $E_{\infty}$ which lands in a symplectisation component of $X^k_{\infty}$ is a cylinder on its asymptotic Reeb orbit.
\end{lma}
\begin{proof}
The $\lambda$-energy of a finite-energy curve $F=(a,v):S\setminus Z\rightarrow \RR\times M$ in the symplectisation is non-negative, which entails
\[\int v^*d\lambda\geq 0\]
\noindent By Stokes's theorem this integral is the sum of the periods of the asymptotic Reeb orbits, weighted $\pm 1$ according to whether $F$ is asymptotic to $-\infty\times\gamma$ or $+\infty\times\gamma$. Let $s^\pm$ denote the number of positive (respectively negative) Reeb orbits to which $F$ is asymptotic.

Write $E_{\infty}^{\ell}$ for the part of $E_{\infty}$ landing in $\Sympl_{\ell}(M)$ and let $\ell=m$ be the top level of the symplectisation. We know by what was said above that the positive asymptotics of $E_{\infty}^{m}$ are simple Reeb orbits. Therefore for $F=E_{\infty}^{\ell}$, the inequality above becomes
\[s^+-\sum_{i=1}^{s^-}\cov(\gamma^-_i)\geq 0\]
\noindent and in particular, $s^+\geq s^-$.

Now $E_{\infty}^V$ has a single component, so if $s^+>1$ then $E_{\infty}^{m}$ will connect two of the negative asymptotic orbits of $E_{\infty}^V$, but the genus of $E_{\infty}$ is zero. Hence $s^+=1$ and so $s^-\leq 1$. Since $M\cong\RR\mathbb{P}^3$ and a single Reeb orbit represents a nontrivial element of $\pi_1(\RR\mathbb{P}^3)\cong\ZZ/2$, $s^-\neq 0$ or else $E_{\infty}^{m}$ would be a nullhomotopy of the Reeb orbit. Hence $s^-=1$ and $E_{\infty}^{m}$ is a cylinder with zero energy. This implies that it is a Reeb cylinder as claimed.

Inductively applying this argument to the lower symplectisation levels allows us to deduce the lemma for all $\ell$.
\end{proof}

Finally, we consider $E_{\infty}^W$. Topologically, it must consist of finite-energy planes in order for the building $E_{\infty}$ to have genus zero. Since $E_{\infty}^V$ has simple negative asymptotics and the intermediate levels $E_{\infty}^{\ell}$ are cylindrical, the positive asymptotics of $E_{\infty}^W$ are also simple. In lemma \ref{mustbealphabeta}, we classified finite-energy planes with simple asymptotics in $W$. They were either $\alpha$- or $\beta$-planes. In summary:

\begin{prp}\label{finalexceptional}
If $E(J_{t_j})$ is a convergent sequence of $J_{t_j}$-holomorphic exceptional spheres for a neck-stretch $J_t$ then the limit building consists of:

\begin{itemize}
\item $E_{\infty}^V$, a connected punctured finite-energy $J_{\infty}$-holomorphic sphere in $V$ asymptotic to a finite collection of simple Reeb orbits $\{\gamma_i\}$ on $M$,
\item $E_{\infty}^W$, a collection of $\alpha$- or $\beta$-planes in $W$ asymptotic to the Reeb orbits $\{\gamma_i\}$.
\end{itemize}
\end{prp}

The components of the buildings in symplectisation levels would necessarily be cylindrical, but these are ruled out by the definition of stability for a holomorphic building since there are no marked points.

Notice that since $E_1$ and $E_2$ have homological intersection zero, their $J_{\infty}$-holomorphic representatives have no isolated intersections. Otherwise, for large $t$, there would be positively intersecting $E_1(J_t)$ and $E_2(J_t)$. Since the $W$-parts of their limits must be non-empty (as each class $E_i$ has non-trivial homological intersection with $L\subset W$), and consist of $\alpha$- and $\beta$- planes, we deduce:

\begin{lma}
The punctured curves $E_{1,\infty}^W$ and $E_{2,\infty}^W$ consist of $\alpha$- and $\beta$-planes all of which are asymptotic to \emph{the same Reeb orbit $R$}.
\end{lma}
\begin{proof}
Let $A$ be an $\alpha$-plane which is a component of $E_{1,\infty}^W$ and $B$ a $\beta$-plane which is a component of $E_{2,\infty}^W$. Such components must exist since $E_1\cdot L=-1$ and $E_2\cdot L=1$. If $A$ and $B$ were not asymptotic to the same Reeb orbit then they would have an isolated point of intersection which is disallowed as remarked above. Therefore they have a common asymptotic orbit $R$. Similarly, any $\beta$-plane in $E_{1,\infty}^W$ (respectively $\alpha$-plane in $E_{2,\infty}^W$) which was not asymptotic to $R$ would show up as a self-intersection of $E_1(J_t)$ (respectively $E_2(J_t)$) for large $t$. But by the adjunction formula, we saw these closed curves were embedded.
\end{proof}

For the other exceptional classes $E_3$, $E_4$ and $S_{34}$, the same analysis carries through except that the $W$-part may be empty. Again by positivity of intersections, any $\alpha$- or $\beta$- plane in $E_3^W$, $E_4^W$ or $S_{34}^W$ must be asymptotic to $R$.

\begin{prp}
$E_{1,\infty}^W$ and $E_{2,\infty}^W$ consist of a single $\beta$- and $\alpha$-plane respectively. $E_{3,\infty}$, $E_{4,\infty}$ and $S_{34,\infty}$ have no $W$-component.
\end{prp}
\begin{proof}
Let $E$ stand for any of these classes. We have noted that $E^W_{\infty}$ (if non-empty) consists of $\alpha$- or $\beta$-planes asymptotic to a given Reeb orbit $R$. Therefore $E^V_{\infty}$ has all its asymptotic orbits equal to $R$. Suppose $E^V_{\infty}$ has $k$ punctures. Suppose $k>1$. By remark \ref{transvevenergy1}, the evaluation map from the moduli space of $E^V_{\infty}$ to the product $\rho^k$ is transverse to the various strata of the multi-diagonal for generic neck-stretching data $J_1$. Since $\rho$ has dimension 2, the multi-diagonal has top strata of codimension 2. But the moduli space is zero-dimensional and $\rho^k$ has dimension $2k>2$ so this transversality means that the image of evaluation map is disjoint from the multi-diagonal. This contradicts the fact that all asymptotic orbits are equal to $R$. Therefore $k=0$ or $1$. If $k=1$ then $E$ must have intersection number $\pm 1$ with $L$ (as $E_{\infty}^W$ consists of a single $\alpha$- or $\beta$-plane). If $k=0$ then $E$ must have zero intersection with $L$. Since the classes $E_3$, $E_4$, $S_{12}$ and $S_{34}$ have zero intersection with $L$ they fall into this category, while the classes $E_1$ and $E_2$ must have $k=1$.
\end{proof}

The final point is to show that $S_{12,\infty}$ has no $W$-part.

\begin{lma}
$S^W_{12,\infty}=\emptyset$.
\end{lma}
\begin{proof}
$S^W_{12,\infty}$ consists of a union of $\alpha$- and $\beta$-planes. Each of these contributes $\pm 1$ respectively to the intersection number $S_{12}\cdot [L]$. Since this intersection number vanishes, if $S^W_{12,\infty}\neq\emptyset$ then it must contain at least one $\alpha$-plane and one $\beta$-plane. If these were asymptotic to different Reeb orbits then they would intersect and for large $T$, $S_{12}(J_T)$ would not be embedded, contradicting lemma \ref{simpleenergy1}. However, a transversality argument like the one in the preceding lemma will rule out the possibility that $S^V_{12,\infty}$ has two punctures asymptotic to the same Reeb orbit. Therefore the $W$-part must be empty.
\end{proof}

\subsection{The classes $H-E_3$, $H-E_4$}

The purpose of this section is to prove that (for generic neck-stretching data) if $i=3,4$ there is a smooth $J_t$-holomorphic curve homologous to $H-E_i$ disjoint from $L$.

Let $u_j$ be a sequence of $J_{t_j}$-holomorphic curves in the homology class $H-E_i$ in $X$. By SFT compactness, there is a Gromov-Hofer convergent subsequence $u_{j_k}$ whose limit is a building $F$. We note some properties of the $V$-part $F^V$ of such a building:

\begin{lma}\label{propertiesofvparts}
$F^V$ is non-empty and falls into one of three categories $\mbox{\textcircled{a}}$, $\mbox{\textcircled{b}}$ and $\mbox{\textcircled{c}}$:
\[ 
\xy 
(-3,5)*{\mbox{\textcircled{a}}};
(-3,0)*{};(7,0)*{}; **\crv{(2,10)};
(2,0)*{2};
(7,5)*{1\times};
(20,5)*{\mbox{\textcircled{b}}};
(20,0)*{};(30,0)*{}; **\crv{(25,10)};
(25,0)*{1};
(30,5)*{1\times};
(32,0)*{};(42,0)*{}; **\crv{(37,10)};
(37,0)*{1};
(42,5)*{1\times};
(55,5)*{\mbox{\textcircled{c}}};
(55,0)*{};(65,0)*{}; **\crv{(60,10)};
(60,0)*{1};
(65,5)*{2\times};
\endxy 
\]
\noindent where $\xy(0,0)*{};(10,0)*{}; **\crv{(5,10)};
(5,0)*{A};
(10,5)*{n\times};\endxy$ stands for a component which is an $n$-fold cover of a curve with $c_1^{\Phi}=A$.\end{lma}
\begin{proof}
The maximum principle says that $W$ contains no closed holomorphic curves, hence $F^V$ is nonempty. Arguing as in the proof of lemma \ref{propertiesexceptional}, the relative first Chern class $c_1^{\Phi}(F^V)$ is $2$ and $F^V$ has at most two connected components. Since $1$ is the minimal Chern number, multiple covering can only occur if the limit is a double cover of a simple curve with $c_1^{\Phi}=1$. These are precisely the cases listed in the lemma.
\end{proof}

Since there is a countable set of possible topologies and homology classes for simple genus 0, curves with $c_1^{\Phi}=1,\ 2$, it can be ensured as in remarks \ref{regularityforexceptional} and \ref{transvevenergy1} that all such moduli spaces are smooth and of the expected dimension. The dimension formula for classes with $c_1^{\Phi}=2$ becomes
\begin{eqnarray*}
\dim(\mM^A_{0,s^-}(J_{\infty}|_V)) & = & 2+2s^--2\sum_{i=1}^{s^-}\cov(\gamma_i) \\
& \leq & 2
\end{eqnarray*}
\noindent which is only positive if
\begin{enumerate}
\item all asymptotic Reeb orbits are simple, in which case the expected dimension is 2, or
\item all but one of the asymptotic Reeb orbits are simple, the one exception being a doubly-wrapped orbit. The expected dimension in this case is 0.
\end{enumerate}

This gives the following breakdown of the three cases from lemma \ref{propertiesofvparts}:

\begin{tabular}{ccl}
Dimension 0: & \textcircled{a}: & one double orbit \\
 & \textcircled{b}: & all simple orbits\\
 & \textcircled{c}: & all simple orbits \\
Dimension 2: & \textcircled{a}: & all simple orbits \\
\end{tabular}

\noindent where we understand a multiply-covered curve to live ``in a $0$-dimensional moduli space'' if its underlying simple curve has a moduli space of dimension $0$. We call curves of type \textcircled{a} with all simple orbits \emph{good curves} and all other curves \emph{bad curves}.

Let $J_{t_j}$ be a sequence for which the curves $E_k(J_{t_j})$ and $S_{k\ell}(J_{t_j})$ Gromov-Hofer converge to $J_{\infty}$-holomorphic limit buildings for all $i,k,\ell$.

\begin{lma}
There exists a good curve in $V$ which occurs as the $V$-part of the Gromov-Hofer limit of a sequence of $J_{t_j}$-holomorphic curves in $X$ homologous to $H-E_i$ (for $i=3,4$).
\end{lma}
\begin{proof}
There is a dense set of points in $V$ which do not lie on a bad curve, since there are countably many bad curves. Pick such a point, $x$. For every $t_j$ there is a $J_{t_j}$-holomorphic curve homologous to $H-E_i$ passing through $x$. A subsequence of these Gromov-Hofer converges to a building in $X_{\infty}^k$ through $x$, whose $V$-part must then be a good curve.
\end{proof}

Once we have this Gromov-Hofer limit, $C_x$, whose $V$-part is a good curve, we examine its asymptotic Reeb orbits. These are simple by construction and by the same argument as in lemma \ref{reebcyl} the parts of $C_x$ which land in the symplectisation of $M$ are just Reeb cylinders and the components of $C_x^W$ are $\alpha$- and $\beta$-planes.

\begin{lma}
$C^W_x$ is empty for generic neck-stretching data.
\end{lma}
\begin{proof}
Suppose that $C_x^W$ were non-empty. Since $H-E_i$ has homological intersection 0 with $E_1$ and $E_2$, and since $E_{1,\infty}^W$ and $E_{2,\infty}^W$ consist of a single $\alpha$- and a single $\beta$-plane each asymptotic to the same Reeb orbit $R$, the components of $C_x^W$ must also be asymptotic to $R$ in order to avoid intersections. Because $H-E_i$ has zero intersection with $L$, there must be at least one of each type of plane (recall that $C_x^W$ consists of $\alpha$- and $\beta$-planes), so $C_x^V$ has $k\geq 2$ punctures. The evaluation map from the moduli space of $C_x^V$ to $\rho^k$ is transverse to all strata of the multi-diagonal for generic neck-stretching data, which have codimension at least 2 in $\rho^k$. Since the moduli space of $C_x^V$ has dimension 2, the dimension of $\rho^k$ is $2k$ and $\ev(C_x^V)=(R,\ldots,R)$ this means that $k=2$ and nearby curves $C'$ in the moduli of $C_x^V$ have at least one asymptotic orbit not equal to $R$.

We must show that there is such a nearby curve $C'$ which occurs as the $V$-part of a Gromov-Hofer limit $C'_{\infty}$ of a sequence of $J_{t_j}$-holomorphic curves in $X$ homologous to $H-E_i$. For then, since the asymptotic orbits of $C'$ involve an orbit not equal to $R$, the $W$-part of $C'_{\infty}$ contains a $\beta$- or $\alpha$-plane not asymptotic to $R$ which therefore intersects either $E^W_{1,\infty}$ or $E^W_{2,\infty}$ respectively, contradicting positivity of intersections and the fact that $E_1\cdot (H-E_i)=0$ and $E_2\cdot (H-E_i)=0$.

\begin{lma}\label{gettingcloser}
There is a nearby curve $C'$ in the moduli of $C_x^V$ which occurs as the $V$-part of the Gromov-Hofer limit of a sequence of $J_{t_{j_k}}$-holomorphic curves in $X$ homologous to $H-E_i$ for a subsequence $j_k\subset j$.
\end{lma}
\begin{proof}
Let $x_m$ be a sequence of points tending to $x$. For $x_1$ there is a subsequence $j^1_k\subset j$ such that the $J_{t_{j^1_k}}$-holomorphic curves in $X$ homologous to $H-E_i$ passing through $x_1$ Gromov-Hofer converge. Similarly, extract a subsequence $j^2_k\subset j^1_k$ of this for curves passing through $x_2$ and (after iteratively constructing a subsequence $j^m_k$ for each $m$) extract a Cantor diagonal subsequence $\delta_m=j^m_m$ for which all sequences of $J_{\delta_m}$-holomorphic curves homologous to $H-E_i$ and passing through some fixed $x_m$ Gromov-Hofer converge to a building $A_m$. Let $C_m$ denote the $V$-part of $A_m$. We must show that as $m\rightarrow\infty$ there is a subsequence of these buildings which Gromov-Hofer converge to $C_x^V$, for then we may take $C'=C_M$ for large $M$.

There is certainly a Gromov-Hofer convergent subsequence $C_{m_k}$ with limit $C_{m_{\infty}}$ whose $V$-part is $C^V_{m_{\infty}}$. Since $C^V_x$ and $C_{m_{\infty}}^V$ intersect positively at $x$, they must be geometrically indistinct, for otherwise there would be a large $k$ for which the $J_{\delta_k}$-holomorphic curve homologous to $H-E_i$ through $x_k$ and the curve through $x$ would intersect with positive local intersection number, contradicting the fact that $(H-E_i)\cdot(H-E_i)=0$.

Since $x$ does not lie on a bad curve, the asymptotics of $C_{m_{\infty}}^V$ are all simple Reeb orbits and the symplectisation parts of $C_{m_{\infty}}$ are all Reeb cylinders. Since none of the curves in question had any marked points, such cylinders would necessarily be unstable, so cannot occur. Hence $C_{m_k}$ Gromov-Hofer converges to $C_x$ and for large $k$, $C_{m_k}$ lives in the same moduli space as $C_x$.
\end{proof}

This proves the lemma.

\end{proof}

\subsection{The class $H$}

We finally turn our attention to Gromov-Hofer limits $F$ of stable curves in the homology class $H$ under neck-stretching.

\begin{lma}\label{pretty2}
$F^V$ is non-empty and falls into one of five categories:
\[ 
\xy 
(-3,5)*{\mbox{\textcircled{a}}};
(-3,0)*{};(7,0)*{}; **\crv{(2,10)};
(2,0)*{3};
(7,5)*{1\times};
(20,5)*{\mbox{\textcircled{b}}};
(20,0)*{};(30,0)*{}; **\crv{(25,10)};
(25,0)*{1};
(30,5)*{1\times};
(32,0)*{};(42,0)*{}; **\crv{(37,10)};
(37,0)*{2};
(42,5)*{1\times};
(55,5)*{\mbox{\textcircled{c}}};
(55,0)*{};(65,0)*{}; **\crv{(60,10)};
(60,0)*{1};
(65,5)*{1\times};
(67,0)*{};(77,0)*{}; **\crv{(72,10)};
(72,0)*{1};
(77,5)*{1\times};
(79,0)*{};(89,0)*{}; **\crv{(84,10)};
(84,0)*{1};
(89,5)*{1\times};
(-3,-5)*{\mbox{\textcircled{d}}};
(-3,-10)*{};(7,-10)*{}; **\crv{(2,0)};
(2,-10)*{1};
(7,-5)*{3\times};
(20,-5)*{\mbox{\textcircled{e}}};
(20,-10)*{};(30,-10)*{}; **\crv{(25,0)};
(25,-10)*{1};
(30,-5)*{1\times};
(32,-10)*{};(42,-10)*{}; **\crv{(37,0)};
(37,-10)*{1};
(42,-5)*{2\times};
\endxy 
\]
\noindent where $\xy(0,0)*{};(10,0)*{}; **\crv{(5,10)};
(5,0)*{A};
(10,5)*{n\times};\endxy$ stands for a component which is an $n$-fold cover of curve with $c_1^{\Phi}=A$.\end{lma}
\begin{proof}
Non-emptiness follows from the maximum principle. The possible degenerations are a simple consequence of the fact that $c^{\Phi}_1(F^V)=3$.
\end{proof}

We begin with a sequence $J_{t_j}$ for which the curves $E_1(J_{t_j})$, $E_2(J_{t_j})$, $S_{12}(J_{t_j})$ and $C^i_x(J_{t_j})$ Gromov-Hofer converge to $J_{\infty}$-holomorphic buildings $E_{1,\infty}$, $E_{2,\infty}$, $S_{12,\infty}$ and $C^i_{\infty}$ (for $i=3,4$) such that $C^{i,W}_{\infty}=\emptyset$. Here $C^i_x(J_{t_j})$ are the unique $J_{t_j}$-holomorphic curves in the homology classes $H-E_i$ ($i=3,4)$ passing through the point $x$. The aim is to find a subsequence $J_{t_{j_k}}$, a point $z\in V$ and a $J_{t_j}$-complex line in $T_z\overline{X}$ such that:

\begin{itemize}
\item $z\not\in\Xi(J_{t_{j_k}})$ for any $k$ (where $\Xi(J)$ is the union of the exceptional spheres $E_i(J)$),
\item the sequence $H_z^{\ell}(J_{t_{j_k}})$ of stable $J_{t_{j_k}}$-holomorphic curves through $z$ tangent to $\ell$ Gromov-Hofer converge to a $J_{\infty}$-holomorphic building $H_{\infty}$ whose $W$-part is empty.
\end{itemize}

Note that it makes sense to talk about $\ell$ being a $J_{t_j}$-complex line for all $j$, as the neck-stretch only affects the complex structure on the neck so $J_t|_V=J_0|_V$.

\begin{lma}
Suppose $H_{\infty}$ is the Gromov-Hofer limit of a convergent sequence of stable $J_{t_{j_k}}$-holomorphic curves $H_z^{\ell}(J_{t_{j_k}})$ through $z$ and tangent to $\ell$. To prove that the $W$-part $H_{\infty}^W$ is empty it suffices to prove that $H^V_{\infty}$ is connected with simple asymptotic orbits distinct from $R$.
\end{lma}
\begin{proof}
By a similar argument to proposition \ref{finalexceptional} $H_{\infty}^W$ (if it were non-empty) would consist of $\alpha$- and $\beta$-planes which would intersect $E_{1,\infty}$ or $E_{2,\infty}$ and thereby contradict the fact that $H\cdot E_1=H\cdot E_2=0$.
\end{proof}

\begin{rmk}\label{evalmapregular2}
There is a Baire set of neck-stretching data such that the following is true. If $H_{\infty}$ is a $J_{\infty}$-holomorphic building obtained as the Gromov-Hofer limit of a sequence of $J_{t_j}$-holomorphic curves in the homology class $H$ and $H^V_{\infty}$ is connected with simple Reeb asymptotics then any nearby $H'_{\infty}$ in the moduli space of $H^V_{\infty}$ has an asymptotic Reeb orbit different from $R$.
\end{rmk}

\begin{lma}
For a Baire set of neck-stretching data there exists a $J_1$-complex line in the tangent space at $x$ and a subsequence $J_{t_{j_k}}$ such that the sequence of $J_{t_k}$-holomorphic curves $H_x^{\ell}(J_{t_{j_k}})$ Gromov-Hofer converges to a curve $H_{\infty}$ such that $H^V_{\infty}$ is connected and has simple Reeb asymptotics distinct from $R$.
\end{lma}
To prove the lemma, we will begin by naively finding a Gromov-Hofer sequence whose limit is connected with simple Reeb asymptotics, then we will use the remark above to find a nearby $H'_{\infty}$ which has asymptotics distinct from $R$ and finally observe that this punctured curve arises as the $V$-part of a Gromov-Hofer limit of curves as described by the lemma.
\begin{proof}
As in remark \ref{regularityforexceptional}, $J_1$ can be chosen generically so that all moduli spaces of simple finite-energy punctured curves in $V$ are regular and of the expected dimension. The index formula for curves with $c_1^{\Phi}=3$ (such as the $V$-parts of our limits) is

\begin{eqnarray*}
\dim(\mM_{0,s^-}^A(J_{\infty}|_V)) & = & 4+2s^--2\sum_{i=1}^{s^-}\cov(\gamma_i) \\
& \leq & 4
\end{eqnarray*}

\noindent which is only nonnegative if:

\begin{enumerate}
\item all asymptotic Reeb orbits are simple, or
\item all but one of the asymptotic Reeb orbits are simple, the one exception being a doubly-wrapped orbit.
\item all but two of the asymptotic Reeb orbits are simple, the two exceptions being either: one simple and one triply-wrapped orbit or two doubly-wrapped orbits.
\end{enumerate}

This gives us the following possible breakdown of the cases from lemma \ref{pretty2}:

\begin{tabular}{ccl}
Dimension 0: & \textcircled{a}: & one triple orbit \\
 & & two double orbits \\
 & \textcircled{b}: & one double orbit for component with $c_1^{\Phi}=2$\\
 & \textcircled{c}: & all simple orbits \\
 & \textcircled{d}: & all simple orbits \\
 & \textcircled{e}: & all simple orbits \\
Dimension 2: & \textcircled{a}: & one double orbit \\
 & \textcircled{b}: & all simple orbits \\
Dimension 4: & \textcircled{a}: & all simple orbits
\end{tabular}

\noindent where we understand a multiply-covered curve to live ``in a 0-dimensional moduli space'' if its underlying simple curve has a moduli space of dimension 0. We call curves of type \textcircled{a} with all orbits simple \emph{good curves} and all other curves \emph{bad curves}. Curves from the first five rows are called \emph{very bad curves}. For a fixed almost complex structure $J_{\infty}|_V$, there is a dense set of points in $V$ which do not lie on a very bad curve. Let $\Omega\subset V$ be a countable, dense set of such points. Let $\Lambda\subset\mathbb{P}^{J_1}(TV|_{\Omega})$ be a set of $J_1$-complex lines such that $\Lambda\cap \mathbb{P}^{J_1}(T_xV)$ is dense in $\mathbb{P}^{J_1}(T_xV)$ for every $x\in\Omega$ (and recall that $J_1(x)=J_t(x)$ for all $t\in[1,\infty]$).

Number the elements of $\Omega=\{x_1,x_2,\ldots\}$. There is a subsequence $t^1_{j_k}\subset t_j$ such that $x_1\not\in\Xi(J_{t^1_{j_k}})$ for all $k$, or else the Gromov-Hofer limit of stable curves through $x_1$ homologous to $H$ would have a component with $c_1^{\Phi}=1$ passing through $x_1$. Similarly there is a subsequence $t^2_{j_{k_m}}\subset t^1_{j_k}$ for which $x_2$ is not in $\Xi(J_{t^2_{j_{k_m}}})$ for any $m$. Iteratively construct a Cantor diagonal subsequence (written $t_j$ for brevity) for which $x_i\not\in\Xi(J_{t_j})$ for all $i$ and $j$.

Proposition \ref{complextangent} now implies that for each $\ell\in\Lambda$ there is a sequence $H_x^{\ell}(J_{t_j})$ of $J_{t_j}$-holomorphic spheres in the homology class $H$ passing through $x$ and tangent to $\ell$. Passing to a Cantor diagonal subsequence we can ensure that these all Gromov-Hofer converge to $J_{\infty}$-holomorphic buildings $H_{x,\infty}^{\ell}$.

These buildings cannot all be bad curves, thanks to the following observation:

\begin{lma}
In $\mathbb{P}^{J_1}(TV)$ the set of points which are complex tangent lines to bad curves have open, dense complement.
\end{lma}

The lemma follows from the dimension formula: the set of such points form a countable union of 2- and 4-dimensional subspaces in the six-dimensional space $\mathbb{P}^{J_1}(TV)$. This subset is closed by SFT compactness.

Once we have this good curve, if it has an asymptotic orbit different from $R$ then we have proved the lemma, setting $H_{\infty}=H_{x,\infty}^{\ell}$. If not, remark \ref{evalmapregular2} implies that any nearby curve in the local moduli space of $H_{x,\infty}^{\ell}$ does have an asymptotic orbit different from $R$. Pick a sequence $\ell_i\in\Lambda\cap\mathbb{P}^{J_1}(T_xV)$ of good curves tending to $\ell$. By the same argument as proved claim \ref{gettingcloser} we deduce that for large $i$ we may take $H_{\infty}=H_{x,\infty}^{\ell_i}$.
\end{proof}

This completes the proof of proposition \ref{grottyanalysis}.

\section{Analysis of limit-buildings: ternary case}\label{prf2}

Let $L$ be a Lagrangian sphere in the ternary homology class $H-E_1-E_2-E_3$ in $\DD_3$ or $\DD_4$. The details of the analysis here are very similar to those in the previous section. The results are that for generic neck-stretching data:

\begin{itemize}
\item All exceptional classes $E_1$, $E_2$, $E_3$, ($S_{12}$, $S_{13}$, $S_{23}$) have limit-buildings consisting of a single $\alpha$- (respectively $\beta$-) plane and $S_{14}$ and $E_4$ have limit-buildings with empty $W$-parts.
\item In $\DD_3$, one can find a point $x\in V$ for which the three sequences of $J_t$-holomorphic curves homologous to $H-E_1$, $H-E_2$ and $H-E_3$ passing through $x$ Gromov-Hofer converge to buildings with empty $W$-part.
\item In $\DD_4$, one can find a point $x$ in $S_{14,\infty}$ for which the $J_t$-holomorphic curves homologous to $H-E_2$ and $H-E_3$ passing through $x$ have Gromov-Hofer limits with empty $W$-part.
\end{itemize}

\section{Proof of theorem \ref{isothm}}\label{isoprf}

Let $(X,\omega)$ be a monotone symplectic Del Pezzo surface $\DD_n$ ($n\leq 4$), $L$ an embedded Lagrangian sphere, $J_1$ a generic choice of neck-stretching data and $\{J_t\}_{t=1}^T$ be the corresponding family of $\omega$-compatible complex structures obtained by stretching the neck around $L$. Extend $J_t$ to a family $\{J_t\}_{t=0}^T$ where $J_0$ is the standard complex structure on $\DD_n$ coming from thinking of it as a $n-1$-point complex blow-up at $(\infty,\infty)$ of $\PP{1}\times\PP{1}$ with its product complex structure.

I will discuss the case of binary Lagrangian spheres in $n=2$. Generalisation of the argument to binary and ternary spheres in $n=3,4$ should be clear.

Fix a point $x$ and a complex direction $\ell$ at $x$ such that for a suitable subsequence $t_j$ the corresponding sequence $H_{x,\ell}(J_{t_j})$ of $J_{t_j}$-holomorphic curves in the homology class $H$ through $x$ tangent to $\ell$ Gromov-Hofer converge to a $J_{\infty}$-holomorphic curve disjoint from $L$. The sequence $S_{12}(J_{t_j})$ also has a Gromov-Hofer convergent subsequence, $j_k\subset j$, whose limit curve is disjoint from $L$. Therefore for $T$ large enough, $S_{12}(J_T)$ and $H_{x,\ell}(J_T)$ are disjoint from $L$.

Extend the family $\{J_t\}_{t=1}^T$ to a family $\{J_t\}_{t=0}^T$ where $J_0$ is an integrable complex structure coming from considering $\DD_2$ as the blow-up of $\PP{2}$ at two points in general position. For all $t\in[0,T]$ there are unique smooth, embedded $J_t$-spheres in the homology classes $E_1$, $E_2$ and $S_{12}$ and the implicit function theorem implies that as $t$ varies, these exceptional spheres undergo smooth isotopy (see \cite{MS04}, remark 3.2.8).

For each $t$ the space $\mathcal{R}_t$ of pairs $\{(x',\ell')\in V\times\mathbb{P}^{J_t}(T_xV)\}$ for which there is a smooth $J_t$-curve homologous to $H$ through $x'$ tangent to $\ell'$ is open, dense and connected in the total space of $\mathbb{P}^{J_t}(TV)$: its complement consists of strata with codimension at least 2. Such curves are regular, so again by remark 3.2.8 in \cite{MS04} there is a path in $\bigcup_{t\in[0,T]}\mathcal{R}_t$ which ends at $(x,\ell,T)$ and starts in $\mathcal{R}_0$. This gives a smooth isotopy of (smooth, regular, embedded) $J_t$-holomorphic curves $H_t$ homologous to $H$, ending with one which is disjoint from $L$.

Now for each $t$ consider the configuration $H_t\cup E_1(J_t)\cup E_2(J_t)\cup S_{12}(J_t)$. Choose families of Darboux balls centred at the points of $S_{12}$ where the curves in the configuration intersect. In each family of balls, perform local perturbations of $H_t$, $E_1(J_t)$ and $E_2(J_t)$ to obtain a configuration of symplectic surfaces which intersect symplectically orthogonally. This perturbation is chosen to leave $H_0$, $E_1(J_0)$ and $E_2(J_0)$ unchanged and the Darboux balls can be chosen so that at $t=T$ they are disjoint from $L$ (since $L$ is disjoint from $S_{12}$). Write $C_t$ for the resulting smooth isotopy of the configuration of perturbed spheres.

Each sphere has a symplectic neighbourhood to which the isotopy extends, by the symplectic neighbourhood theorem, and since the spheres are now symplectically orthogonal the isotopy extends over a neighbourhood of the whole configuration.

Now the existence of a global symplectomorphism $\Psi_t:X\rightarrow X$ for which $\Psi_t(C_0)=C_t$ follows from Banyaga's symplectic isotopy extension theorem because these exceptional spheres generate $H_2(X,\ZZ)$:

\begin{thm}[Banyaga's isotopy extension theorem, see \cite{MS05}, p. 98]\label{banyaga}
Let $(X,\omega)$ be a compact symplectic manifold and $C\subset X$ be a compact subset. Let $\phi_t:U\rightarrow X$ be a symplectic isotopy of an open neighbourhood $U$ of $C$ and assume that
\[H^2(X,C;\RR)=0\]
\noindent Then there is a neighbourhood $\mathcal{N}\subset U$ of $C$ and a symplectic isotopy $\psi_t:X\rightarrow X$ such that $\psi_t|_{\mathcal{N}}=\phi_t|_{\mathcal{N}}$ for all $t$.
\end{thm}

Therefore $L_t=\Psi_t^{-1}(L)$ is an isotopy of $L$ through Lagrangian spheres which disjoins it from $S_{12}(J_0)=\Psi_T^{-1}(S_{12}(J_T))$ and from the $J_0$-holomorphic line $\Psi_T^{-1}(H_{x,\ell}(J_t))$. This proves theorem \ref{isothm} for $\DD_2$.

\section{Acknowledgements}
I would like to thank the Faulkes Foundation for the inspiring support they have given to geometry in the past decade and for the grant which enabled me to complete this work. Thanks also to Ivan Smith, my PhD supervisor, for sharing his time, knowledge and suggestions so generously. Discussions with Jack Waldron, Martin Schwingenheuer and Mark McLean have been invaluable. Particular thanks go to Chris Wendl for explaining to me his transversality results and to the referee of this paper for their detailed and insightful commentary. This paper was inspired by the work of Richard Hind and the philosophy of Eliashberg and Hofer.


\begin{thebibliography}{99}%
\bibitem{BEHWZ03}
 {{F. Bourgeois, Ya. Eliashberg, H. Hofer, K. Wysocki and E. Zehnder}}
 `Compactness results in symplectic field theory',
 {\em Geometry and Topology} (2) 7 (2003) 799--888.
%
\bibitem{Bou02}
 {{F. Bourgeois}}, `A Morse-Bott approach to contact homology', doctoral thesis, Stanford University, 2002.
%
\bibitem{CieMoh05}
 {{K. Cieliebak \and K. Mohnke}}
 `Compactness for punctured holomorphic curves',
 {\em Journal of Symplectic Geometry} (4) 3 (2005) 589--654.
%
\bibitem{Hi03}
 {{R. Hind}}
 `Lagrangian unknottedness in Stein surfaces',
 preprint, 2003;\\ available online at http://www.nd.edu/$\sim$rhind/lag23.pdf
%
\bibitem{Hi04}
 {{R. Hind}}
 `Lagrangian spheres in $S^2\times S^2$',
 {\em Geometric and Functional Analysis} (2) 14 (2004) 303--318.
%
\bibitem{Hof93}
 {{H. Hofer}}
 `Pseudoholomorphic curves in symplectizations with applications to the Weinstein conjecture in dimension three',
 {\em Inventiones Mathematicae} (3) 114 (1993) 515--563.
%
\bibitem{HLS97}
 {{H. Hofer, V. Lizan and J.-C. Sikorav}}
 `On genericity for holomorphic curves in four-dimensional almost-complex manifolds',
 {\em Journal of Geometric Analysis} 7 (1997) 149--159.
%
\bibitem{HWZ96}
 {{H. Hofer, K. Wysocki and E. Zehnder}}
 `Properties of pseudoholomorphic curves in symplectizations IV: Asymptotics with degeneracies',
 {\em Contact and symplectic geometry} (ed C. B. Thomas), Publications of the Newton Institute 8 (Cambridge University Press, 1996) pp. 78--117.
%
\bibitem{Ler95}
 {{E. Lerman}}
 `Symplectic cuts',
 {\em Mathematics Research Letters} 2 (1995) 247--258.
%
\bibitem{McDRat}
 {{D. McDuff}}
 `The structure of rational and ruled symplectic manifolds',
 {\em Journal of the American Mathematical Society} (3) 3 (1990) 679--712.
%
\bibitem{MS04}
 {{D. McDuff and D. Salamon}}
 {\em $J$-holomorphic curves and symplectic topology} (American Mathematical Society, 2004).
%
\bibitem{MS05}
 {{D. McDuff and D. Salamon}}
 {\em Introduction to symplectic topology},
 (Oxford University Press, 2005).
%
\bibitem{Reid}
 {{M. Reid}}
 `Chapters on algebraic surfaces',
 {\em Complex algebraic geometry} (ed J. Koll\'{a}r) IAS/Park City Mathematics Series 3 (American Mathematical Society, 1997) pp. 3--157.
%
\bibitem{SchwarzThesis}
 {{M. Schwarz}}
 `Cohomology operations from $S^1$-cobordisms in Floer homology', doctoral thesis, ETH Z\"{u}rich, 1995.
%
\bibitem{Sei00}
 {{P. Seidel}}
 `Graded Lagrangian submanifolds',
 {\em Bulletin de la Soci\'{e}t\'{e} Math\'{e}matique de France} (1) 128 (2000) 103--149.
%
\bibitem{Sei08}
 {{P. Seidel}}
 `Lectures on four-dimensional Dehn twists',
 {\em Symplectic 4-manifolds and algebraic surfaces} (eds F. Catanese and G. Tian) Lecture Notes in Mathematics 1938 (Springer, 2008) pp.231--268.
%
\bibitem{SS05}
 {{P. Seidel and I. Smith}}
 `The symplectic topology of Ramanujam's surface',
 {\em Commentarii Mathematici Helvetici} 80 (2005) 859--881.
%
\bibitem{Sie07}
 {{R. Siefring}}
 `Relative asymptotic behaviour of pseudoholomorphic half-cylinders',
 {\em Communications on Pure and Applied Mathematics} (12) 61 (2007) 1631--1684.
%
\bibitem{WendlThesis}
 {{C. Wendl}}
 `Finite energy foliations and surgery on transverse links', doctoral thesis, New York University, 2005.
%
\bibitem{Wen08}
 {{C. Wendl}}
 `Automatic transversality and orbifolds of punctured holomorphic curves in dimension 4', arXiv preprint, 2008; arXiv:0802.3842 [math.SG].
\end{thebibliography}
\end{document}